\documentclass[11pt]{article}
\usepackage{amsmath,amsfonts}
\usepackage{txfonts,wasysym,lscape,amssymb,mathrsfs,extarrows}
\usepackage[all,pdf]{xy}
\usepackage[colorlinks,linkcolor=black,anchorcolor=blue,citecolor=green]{hyperref}
\usepackage[titletoc,title]{appendix}
\usepackage[algo2e,vlined,ruled]{algorithm2e}

\usepackage{algorithm}
\usepackage{algorithmic}

 \usepackage{threeparttable}
\usepackage{fancyhdr}

\usepackage{amsfonts}
\usepackage{setspace,url}
\usepackage{graphics,graphicx,epstopdf}
\usepackage{tabularx}
\usepackage{longtable}
\usepackage{booktabs,multirow,array,multicol}
\newcolumntype{C}[1]{>{\centering\arraybackslash}m{#1}}
\usepackage{enumitem}
\usepackage{subfig}
\ifpdf
  \DeclareGraphicsExtensions{.eps,.pdf,.png,.jpg}
\else
  \DeclareGraphicsExtensions{.eps}
\fi

\newtheorem{theorem}{Theorem}[section]
\newtheorem{proposition}{Proposition}[section]
\newtheorem{lemma}{Lemma}[section]
\newtheorem{corollary}{Corollary}[section]
\newtheorem{definition}{Definition}[section]

\newtheorem{remark}{Remark}[section]

\newcommand{\R}{{\mathbb R}}

\newcommand{\be}{\begin{equation}}
\newcommand{\ee}{\end{equation}}

\title{ A Proximal Augmented Lagrangian Method Based on Quadratic Approximations for Weakly Convex Optimization\thanks{Supported by National Key R\&D Program of China under project number 2022YFA1004000, the Major Program of National Natural Science Foundation of China (72192830, 72192831),
 National Natural Science Foundation of China (No. 12371298) and the 111 Project (B16009 ).}}

\author{
Yule Zhang\footnote{
School of   Science, Dalian Maritime University,
Dalian 116085, P. R. China. (ylzhang@dlmu.edu.cn)}, \,\,
Benqi Liu\footnote{Beijing International Center for Mathematical Research, Peking University,
 Beijing 100871, P. R. China. (bqliu @pku.edu.cn)},\,\,
Xiantao Xiao\footnote{Institute of Operations Research and Control Theory, School of Mathematical Sciences, Dalian University
of Technology, Dalian 116024, P. R. China.(xtxiao@dlut.edu.cn)}\,\, and Liwei Zhang\footnote{National Frontiers Science Center for Industrial Intelligence and Systems Optimization, Northeastern University, Shenyang 110819, China;  Key Laboratory of Data Analytics and Optimization for Smart Industry  (Northeastern University),  Ministry of Education, Shenyang 110819, P. R. China. (zhanglw@mail.neu.edu.cn)}}

\date{}

\begin{document}

\maketitle
 \vspace{2mm}
\begin{center}
\parbox{13.5cm}{\small
\textbf{Abstract.}  This paper proposes QPALM, a proximal augmented Lagrangian method based on quadratic approximations, for solving nonlinear programming problems with weakly convex objective and constraint functions. The algorithm is constructed by incorporating quadratic approximations of both the objective and constraint functions into a proximal Lagrangian framework. We establish its non-asymptotic convergence rate in terms of the total number of subproblems solved. The convergence of QPALM is characterized by three metrics associated with the $\varepsilon$-KKT conditions: the squared norm of the gradient of the Moreau envelope of the Lagrangian, the average constraint violation, and the average complementarity violation. All three metrics are shown to converge at a rate of $O(T^{-1/3})$ after $T$ iterations. Preliminary numerical results demonstrate the practical efficiency of the proposed method. These results are established under two mild conditions: (i) weak convexity of all problem functions, and (ii) the existence of a strictly feasible point. The proposed QPALM is a sequentially strongly convex programming method that is readily implementable. \\[10pt]
\textbf{Key words.}  proximal  Lagrangian method, quadratic approximations, $\varepsilon$-KKT conditions, constraint violation, complementarity violation, iteration complexity. \\
\textbf{AMS Subject Classifications(2000):} 90C30. }
\end{center}

\section{Introduction}
\setcounter{equation}{0}
Augmented Lagrangian methods form an important class of approaches for solving nonlinear optimization problems, tracing back to the work of Hestenes (1969) \cite{Hestenes1969} and Powell (1969) \cite{P69}. These methods fall under the framework of sequential unconstrained minimization techniques, where each iteration requires solving an unconstrained minimization problem either exactly or approximately. To address this limitation, many studies have focused on constructing new subproblems based on the augmented Lagrangian function that are easier to solve-such as the linearized augmented Lagrangian method proposed by Wang and Yuan (2015) \cite{Wang2015a} and Wang and Zhang (2015) \cite{Wang2015b}. While the local convergence rate of the classical augmented Lagrangian method is well understood (see, e.g., Bertsekas (1982) \cite{Bertsekas1982}), research on augmented Lagrangian methods with global convergence remains less satisfactory.

Over the past two decades, with the emergence of numerous algorithms in machine learning, researchers have begun evaluating algorithm performance from a different perspective-namely, iteration complexity-which examines the approximation of optimality achieved after $N$ steps. A typical example is the worst-case complexity analysis of an SQP method for nonlinear equality constrained stochastic optimization in \cite{Curtis2024}. A number of studies on augmented Lagrangian methods have emerged in this direction, including \cite{GrapigliaYuan2021}, \cite{XieWright2021}, \cite{HongHajinezhadZhao2017}, \cite{HajinezhadHong2019}, and \cite{Xuyy2025}. A detailed summary of the complexity results-along with the findings of this paper-is provided in Table \ref{tab:complexity_nonconvex}. In this paper, we propose a new augmented Lagrangian method for solving inequality constrained weakly convex optimization problems and analyze its iteration complexity.

\begin{table}[htbp]
\centering
\caption{List of complexity results for some augmented Lagrange-type methods.}
\label{tab:complexity_nonconvex}
\begin{threeparttable}
\scriptsize
\setlength{\tabcolsep}{3pt}
\renewcommand{\arraystretch}{2.15}
\begin{tabularx}{\textwidth}{
>{\centering\arraybackslash}m{2.2cm}
>{\centering\arraybackslash}m{2.7cm}
>{\centering\arraybackslash}m{4.5cm}
>{\centering\arraybackslash}m{2.6cm}
>{\centering\arraybackslash}X}
\toprule
Method & Problem & Optimality measure & Assumptions & Complexity \\
\midrule

ALM \cite{GrapigliaYuan2021}
&
$\begin{aligned}
\min_x\quad & f(x)\\
\mathrm{s.t.}\quad & c_i(x)=0,\ i \in E,\\
& c_i(x)\ge 0,\ i \in I
\end{aligned}$
&
$\begin{aligned}
\|\nabla f(x)&-\sum_{i=1}^m \lambda_i \nabla c_i(x)\| \le \varepsilon,\\
\|c_E(x)\| &\le \varepsilon,\quad
\|c_I^{(-)}(x)\| \le \varepsilon,\\
\lambda_i &\ge 0,\ i\in\mathcal I,\\
\lambda_i &=0\ \text{if } c_i(x)>\varepsilon,\ i\in\mathcal I
\end{aligned}$

&
$f,c_i\in C^1$; BLS; feasible start
&
outer iter.:
$\mathcal O(|\log\varepsilon|)$
or
$\mathcal O\!\left(\varepsilon^{-2/(\alpha-1)}\right)$
\\[10pt]

Prox-AL \cite{XieWright2021}
&
$\begin{aligned}
\min_{x}\quad & f(x) \\
\mathrm{s.t.}\quad & c(x)=0
\end{aligned}$
&
$\begin{aligned}
\|\nabla f(x)+\nabla c(x)\lambda\| &\le \varepsilon,\\
\|c(x)\| &\le \varepsilon
\end{aligned}$
(and)
$\begin{aligned}
d^\top \nabla_{xx}^2L(x,\lambda)d &\ge -\varepsilon\|d\|^2,\ \forall d\in S(x)
\end{aligned}$
&
$f,c_i\in C^2$; compact level set; $\nabla f,\nabla c$ Lipschitz; $\sigma_{\min}(\nabla c)\ge \sigma>0$
&
outer iter.:
$\mathcal O(\varepsilon^{-(2-\eta)})$
\\[10pt]

Prox-PDA \cite{HongHajinezhadZhao2017}
&
$\begin{aligned}
\min_x\quad & f(x)\\
\mathrm{s.t.}\quad & Ax=0
\end{aligned}$
&
$\|\nabla_x L_\beta(x,\mu)\|^2+\|Ax\|^2\le \varepsilon$
&
$\nabla f$ Lipschitz; $A^\top A+B^\top B\succeq I$; $f(x)+\frac{\delta}{2}\|Ax\|^2\ge \underline f$
&
$\mathcal O(1/r)$; equivalently $\mathcal O(\varepsilon^{-1})$
\\[10pt]

PProx-PDA \cite{HajinezhadHong2019}
&
$\begin{aligned}
\min_{x\in X}\quad & f(x)+h(x)\\
\mathrm{s.t.}\quad & Ax=b
\end{aligned}$
&
$\begin{aligned}
\|Ax-b\|^2 &\le \varepsilon,\\
\exists\,\xi\in\partial h(x), \quad &\forall\,x'\in X, \\
\langle \nabla f(x)+A^\top\lambda+\xi,&\ x-x'\rangle \le 0,
\end{aligned}$
&
$\nabla f$ Lipschitz; $h$ convex; $X$ convex compact; $A^\top A+B^\top B\succeq I$
&
$\mathcal O(\varepsilon^{-2})$
\\[20pt]

DPALM \cite{Xuyy2025}
&
$\begin{aligned}
\min_{x\in \mathbb R^d}\quad & F(x)=f(x)+h(x)\\
\mathrm{s.t.}\quad & Ax=b,\,g(x)\leq 0
\end{aligned}$
&
$\begin{aligned}
{\rm dist}\,(0, \partial F(x)+{\cal J}g(x)^Tz+A^Ty)&\leq \varepsilon\\
\sqrt{\|Ax-b\|^2+\|g(x)_+\|^2} &\le \varepsilon,\\
z\geq 0,\,\displaystyle \sum_{i=1}^m |z_ig_i(x)|&\leq \varepsilon,
\end{aligned}$
&
$f$ weakly convex; $h$,$g_i,i=1,\ldots$ convex
&
outer iter.:
$\mathcal O(\varepsilon^{-2})$
\\[10pt]

QPALM (this paper)
&
$\begin{aligned}
\min_{x\in X}\quad & f(x) \\
\mathrm{s.t.}\quad & g_i(x)\le 0,\ i\in\mathcal I
\end{aligned}$
&
$\begin{aligned}
\|\alpha [x-\Pi_X(x-\alpha^{-1}\nabla_x L(x,\lambda))]\|^2 &\leq \varepsilon,    \\
 -\langle \lambda, g(x)\rangle  \leq \varepsilon, \,\,g(x)\leq \varepsilon \textbf{1}_p, \,\,\lambda &\geq -\varepsilon \textbf{1}_p.
\end{aligned}$
&
$f$ and $g_i,i\in {\cal I}$ smooth weakly convex; $X$ convex compact, Slater CQ
& outer iter.:
$\mathcal O(\varepsilon^{-3})$
\\

\bottomrule
\end{tabularx}
\begin{tablenotes}[flushleft]
\footnotesize
\item BLS = bounded level set. The optimality measures are written in a unified style for ease of comparison.
\end{tablenotes}
\end{threeparttable}
\end{table}

In this paper, we consider the following nonlinear programming problem
\begin{equation}\label{eq:1}
\begin{array}{rl}
\displaystyle \min_{x \in X} & f(x)\\[4pt]
{\rm s.t.} & g_i(x) \leq 0, i=1,\ldots,p.\\
\end{array}
\end{equation}
Here $X \subset \mathbb R^n$ is a nonempty  closed convex set, and $f: {\cal O}_0 \rightarrow \mathbb R$, $g_i:{\cal O}_0  \rightarrow \mathbb R$, $i=1,\ldots, p$, where ${\cal O}_0 \subset \mathbb R^n$ is an open convex set containing $X$. The Lagrangian of Problem (\ref{eq:1}) is defined by
\begin{equation}\label{eq:lag}
L(x,\lambda)=f(x)+\displaystyle \sum_{j=1}^p \lambda_j g_j(x), \quad (x, \lambda) \in {\cal O}_0 \times \mathbb R^p.
\end{equation}
We say $(x^*,\lambda^*)\in {\cal O}_0 \times \mathbb R^p$ satisfies the KKT conditions of Problem (\ref{eq:1}) if
\begin{equation}\label{eq:KKToriginalP}
\left\{
\begin{array}{l}
0\in \nabla_x L(x^*,\lambda^*)+N_{X}(x^*),\\[6pt]
0\geq g(x^*) \bot \lambda^* \geq 0,
\end{array}
\right.
\end{equation}
where $N_{X}(x^*)$ is the normal cone of $X$ at $x^* \in X$ in the sense of convex analysis, see \cite{Rock70}. The conditions (\ref{eq:KKToriginalP}) are equivalent to
\begin{equation}\label{eq:KKTref}
\left\{
\begin{array}{l}
\|R_{\alpha}(x^*,\lambda^*)\|=0,\\[6pt]
0\geq g(x^*) \bot \lambda^* \geq 0,
\end{array}
\right.
\end{equation}
where
$$
R_{\alpha}(x,\lambda)=\alpha [x-\Pi_X(x-\alpha^{-1}\nabla_x L(x,\lambda))],
$$
where $\Pi_X(z)$ denotes the Euclidean projection of $z$ onto $X$. This leads us to define $\varepsilon$-approximate KKT point.
\begin{definition}\label{eps-KKT}
 We say $(x,\lambda)$ is a $\varepsilon$-approximate KKT point if the following conditions hold:
\begin{equation}\label{eq:KKTrefa}
\left\{
\begin{array}{l}
\|R_{\alpha}(x,\lambda)\|\leq \varepsilon, -\langle \lambda, g(x)\rangle  \leq \varepsilon \\[6pt]
 g(x)\leq \varepsilon \textbf{1}_p, \lambda \geq -\varepsilon \textbf{1}_p.
\end{array}
\right.
\end{equation}
\end{definition}

The quadratically constrained quadratic programming (QCQP) approximation for Problem (\ref{eq:1})  at $x^t$ is defined as
\begin{equation}\label{eq:LPt}
\begin{array}{ll}
\min & q^t_0(x) \\[4pt]
{\rm s.t.} &  q^t_i(x)  \leq 0,i=1,\ldots,p,\\[4pt]
 & x\in X,
\end{array}
\end{equation}
where $q^t_j(x),j=0,1,\ldots,p$ are quadratic approximations of $f(x)$ and $g_j(x)$, $j=1,\ldots,p$ at $x^t$, respectively. Functions $q^t_j(x),j=0,1,\ldots,p$  are defined by
\begin{equation}\label{eq:qs}
\begin{array}{l}
q^t_0(x)= f(x^t)+\langle \nabla_xf(x^t),x-x^t \rangle+\displaystyle \frac{1}{2}\langle \Sigma^t_0(x-x^t),x-x^t\rangle\\[6pt]
q^t_i(x)= g_i(x^t)+ \langle \nabla_xg_i(x^t),x-x^t \rangle +\displaystyle \frac{1}{2}\langle \Sigma^t_i(x-x^t),x-x^t\rangle,\,\,
i=1,\ldots,p.
\end{array}
\end{equation}
The augmented Lagrange function for Problem (\ref{eq:LPt})
is defined by
\begin{equation}\label{augL}
{\cal L}^t_{\sigma}(x,\lambda): =q_0^t(x) +\displaystyle \frac{1}{2\sigma}\left[ \sum_{i=1}^p[\lambda_i+\sigma q^t_i(x)]_+^2-\|\lambda\|^2\right]
\end{equation}
for $(x,\lambda)\in \mathbb R^n\times \mathbb R^p$. Then  proximal Lagrangian method for Problem (\ref{eq:1}) based on quadratic approximations may be described as follows.\mbox{}\\[4pt]
{\bf QPALM}: A quadratic  approximation based  proximal augmented Lagrangian method
\begin{description}
\item[Step 0 ] Input $\lambda^1=0 \in \mathbb R^p$, $x^1 \in \mathbb R^n$, a positive  integer $N$. Positive parameters $\sigma$ and $\alpha$. Set $t:=1$.
\item[ Step 1] Set
 \begin{equation}\label{xna}
\begin{array}{l}
x^{t+1}= \displaystyle\hbox{arg}\min \,\left\{ {\cal L}^t_{\sigma }(x,\lambda^t) +\displaystyle\frac{\alpha}{2}\|x-x^t\|^2,x \in X\right\}\\[5mm]
\lambda_i^{t+1}=[\lambda_i^t+\sigma q^t_i(x^{t+1})]_+,\,\,i=1,\ldots,p.
\end{array}
\end{equation}
\item[ Step 2] Set $t:=t+1$  and go to Step 1.
\end{description}
In the above algorithm, $[y]_+=\Pi_{\mathbb R^p_+}[y]$ denotes  the projection of $y$ onto $\mathbb R^p_+$ for any $y \in \mathbb R^p$. As far as we are concerned, the main contributions of this paper are summarized as follows.
\begin{itemize}
\item If we set $\sigma = T^{-2/3}$ and $\alpha = 16\gamma_5 T^{1/3}$ in QPALM, where $\gamma_5$ is a constant, then after $T$ iterations, the squared norm of the gradient of the Moreau envelope of the Lagrangian, the average constraint violation, and the average complementarity violation converge to zero at a rate of $O(T^{-1/3})$.
\item At each iteration $t$, the subproblem minimizes a strongly convex function over $X$ and can be solved efficiently using Nesterov's accelerated projected gradient method.
\end{itemize}
The remainder of this paper is organized as follows. In Section 2, we develop properties of QPALM that will be used to analyze Lagrangian gradient violation, constraint violation, and complementarity violation. In Section 3, we establish the iteration complexity of QPALM in terms of the total number of subproblems solved (each subproblem is a strong convex optimization problem with constraint set $X$). In Section 4, we discuss how to solve the subproblems of QPALM and report numerical results from its implementation. We conclude with a summary and discussion in Section 5.

\section{Properties of QPALM}\label{sec-3}
\setcounter{equation}{0}
 In this section, we first list the assumptions of Problem (\ref{eq:1}),  the assumptions about  parameters in QPALM, then the properties of QPALM. Let $\Phi$ be the feasible region of Problem (\ref{eq:1}):
$$
\Phi=\left\{x\in X: g_i(x) \leq 0, i=1,\ldots,p\right\}.
$$
We make the following assumptions about problem functions, which will be used in somewhere.
\begin{description}
    \item[(A1)]Let $D_0>0$ such that
        $$ \|x-z\|\leq D_0, \forall x,z \in X.
        $$

\item[(A2)] Let  $\nu_g>0$ such that
                $$
         \|g(x)\| \leq \nu_g, \forall x \in {\cal O}_0.
        $$
        \item[(A3)]  Let $\kappa_f>0$ and $\kappa_g>0$ such that
                $$
        \|\nabla_x f(x)\| \leq \kappa_f, \,\, \|\nabla_x g_i(x)\| \leq \kappa_g,i=1,\ldots, p, \forall x \in {\cal O}_0.
$$
                \item[(A4)] There exist $\epsilon_0>0$ and $\widehat x \in X$ such that
        $$
        g_i(\widehat x) \leq -\epsilon_0, \,\, i =1,\ldots, p.
$$
         \item[(A5)]There are positive numbers $L_i>0$,$i=0,\ldots,p$ such that $f$ is $L_0$-weakly convex and $g_i$ is $L_i$-weakly convex for $i=1,\ldots, p$.
\end{description}
 We also need the following assumptions about parameters in the algorithm.
\begin{description}
\item[(B1)] Assume that $\Sigma^t_i$ is negatively semidefinite and $\|\Sigma^t_i\|\leq \kappa_{H}$ for $i=1,\ldots, p$, for some positive constant $\kappa_H>0$. We define $C_{\Sigma} := \sqrt{p} \kappa_{H}$.
\item[(B2)] Assume that $q^t_i(x)\leq g_i(x)$ for $i=1,\ldots,p$.
\item[(B3)] Assume that $\Sigma_{0}^{t} \in \mathbb{S}^{n}$ is updated dynamically as $\Sigma_{0}^{t} =-\displaystyle  \sum_{i=1}^{p} \lambda_{i}^{t} \Sigma_{i}^{t} + I$. Thus, its norm is bounded by an increasing function of the multiplier's norm: $||\Sigma_{0}^{t}|| \le \kappa_{\Sigma}(M) := 1 + C_{\Sigma} M$, provided $||\lambda^{t}|| \le M$.
\item[(B4)] Assume that ${\cal L}^t_{\sigma }(x,\lambda^t)$ is a convex function for every $t \in \textbf{N}$.
\item[(B5)] Assume that the problem parameters satisfy the strict regularity condition:
$$ \rho := \frac{D_0^2 C_{\Sigma}}{\epsilon_0} < 1. $$
This condition essentially requires the Slater condition modulus $\epsilon_0$ to be sufficiently large relative to the domain size and the curvature of the constraints.
\end{description}
\begin{remark}\label{remarkB}
If $f$ and $g_i$, $i=1,\ldots,p$ are continuous differentiable over ${\cal O}_0$, then the
conditions (A2) and (A3) are satisfied if Assumption (A1) holds. Conditions (B1)--(B4) are not restricted conditions when $f$ and $g_i$,$i=1,\ldots, p$ satisfy Assumptions (A1)--(A5). We will show, in Section \ref{Sec3},  how to construct $\Sigma^t_i, i=0,1,\ldots, p$ to satisfy
conditions (B1)--(B4).

Furthermore, since $\Sigma_0^t$ depends dynamically on $\lambda^t$ by Assumption (B3), its norm bound $\kappa_\Sigma(M)$ is a function of the multiplier's upper bound $M$. Consequently, the parametric bounds derived in the following lemmas will explicitly depend on $M$ whenever assuming $||\lambda^t|| \le M$.
\end{remark}

Now we develop properties of QPALM, which will be used in the analysis of Lagrange gradient, and constraint violation and complementarity violation. The following  lemma will be used several times in the sequel.
\begin{lemma}\label{lem:opt-x}
 Suppose that  $\Sigma^t_0\in \mathbb S^n$ is positively definite such that Assumption (B4) holds. Then for any $z\in  X$, we have
 \begin{equation}\label{eq:opt-x-1}
\begin{array}{ll}
\displaystyle\langle \nabla_x f(x^t),x^{t+1}-x^t \rangle  +\displaystyle \frac{1}{2}\langle \Sigma^t_0(x^{t+1}-x^t),x^{t+1}-x^t\rangle+ \frac{1}{2\sigma}\|\lambda^{t+1}\|^2
+ \frac{\alpha}{2} \|x^{t+1}-x^t\|^2  \\[5pt]
\leq \displaystyle\langle \nabla_x f(x^t),z-x^t \rangle +\displaystyle \frac{1}{2}\langle \Sigma^t_0(z-x^t),z-x^t\rangle \\[15pt]
\quad\quad + \displaystyle\frac{1}{2\sigma}\left[\displaystyle \sum_{i=1}^p[\lambda^t_i+\sigma (g_i(x^t)+\langle \nabla_x g_i(x^t), z-x^t \rangle)+\displaystyle \frac{1}{2}\langle \Sigma^t_i(z-x^t),z-x^t\rangle]_+^2\right]\\[15pt]
\quad\quad+ \displaystyle\frac{\alpha}{2}(\|z-x^t\|^2-\|z-x^{t+1}\|^2).
\end{array}
\end{equation}
In particular, if we take $z=x^t$, it yields
\begin{equation}\label{eq:opt-x-2}
\begin{array}{ll}
\displaystyle\langle \nabla_x f(x^t),x^{t+1}-x^t \rangle + \frac{1}{2\sigma}\|\lambda^{t+1}\|^2
+ \alpha \|x^{t+1}-x^t\|^2+\displaystyle \frac{1}{2}\|x^{t+1}-x^t\|^2_{\Sigma^t_0}\\[10pt]
\leq  \displaystyle\frac{1}{2\sigma}\left[ \sum_{i=1}^p[\lambda^t_i+\sigma g_i(x^t)]_+^2\right].
\end{array}
\end{equation}
\end{lemma}
{\bf Proof}.
Noting that the minimization problem for defining $x^{t+1}$ in (\ref{xna}) is a strongly convex optimization problem and in view of its optimality conditions, we have that $x^{t+1}$ is also the optimal solution to the following problem
\[
\begin{array}{ll}
\min\limits_{x \in  X}\quad  q^t_0(x)+  \displaystyle\frac{1}{2\sigma} \displaystyle\sum_{i=1}^p\left[\lambda_i^t+\sigma q^t_i(x)\right]_+^2+\displaystyle \frac{\alpha}{2}(\|x-x^t\|^2-\|x-x^{t+1}\|^2).
\end{array}
\]
Then,
the claimed results are obvious.
\hfill $\Box$

In order to give a bound for $\sum_{t=1}^T g_i(x^t)$, we need to estimate an upper bound of $\|x^{t+1}-x^t\|$, which is given in the following lemma.
\begin{lemma}\label{lem:3}
Suppose Assumption (B3)- Assumption (B1) hold and suppose $\|\lambda^t\| \le M$.
If $2\alpha-p(\kappa_g+\kappa_{H}D_0/2)^2\sigma>0$, then
\begin{equation}\label{eq:diffX}
\begin{array}{ll}
\|x^{t+1}-x^t\|
&\leq \displaystyle\frac{2}{2\alpha-p(\kappa_g+\kappa_{H}D_0/2)^2\sigma}
\left(\kappa_f+(\kappa_g+\kappa_{H}D_0/2)\sqrt{p}
[\|\lambda^t\|+\nu_g\sigma]\right).
\end{array}
\end{equation}
If
$\alpha-p(\kappa_g+\kappa_{H}D_0/2)^2\sigma>0$, then
\begin{equation}\label{eq:diffXa}
\begin{array}{ll}
\|x^{t+1}-x^t\|
&\leq \displaystyle\frac{2}{\alpha}
\left(\kappa_f+(\kappa_g+\kappa_{H}D_0/2)\sqrt{p}
[\|\lambda^t\|+\nu_g\sigma]\right).
\end{array}
\end{equation}
\end{lemma}
{\bf Proof}.
Since $x^{t+1}$ is a solution to Problem (\ref{xna}), we have from Assumption (B3) and Assumption (B2) that
$$
\begin{array}{l}
\langle \nabla_x f(x^t), x^{t+1}-x^t\rangle +\displaystyle \frac{1}{2} \left \langle \Sigma^t_0(x^{t+1}-x^t),x^{t+1}-x^t\right \rangle+\displaystyle \frac{1}{2\sigma}\|\lambda^{t+1}\|^2 +\displaystyle \frac{\alpha}{2}\|x^{t+1}-x^t\|^2\\[10pt]
 \leq \displaystyle \frac{1}{2\sigma}\|[\lambda^{t} +\sigma q^t(x^t)]_+\|^2
= \displaystyle \frac{1}{2\sigma}\|[\lambda^{t} +\sigma g(x^t)]_+\|^2
\leq \displaystyle \frac{1}{2\sigma}\|[\lambda^{t} +\sigma g(x^t)]\|^2,
\end{array}
$$
 we have
\[
\alpha \|x^{t+1}-x^t\|^2\leq  \kappa_f\|x^{t+1}-x^t\|+
\frac{1}{2\sigma}\sum_{i=1}^p\left([a_i]_+^2-[b_i]_+^2\right),
\]
in which, for simplicity, we use
\[
a_i:=\lambda_i^t+\sigma g_i(x^t),\quad b_i:=\lambda_i^t+\sigma (g_i(x^t)+\langle \nabla_x g_i(x^t),(x^{t+1}-x^t)\rangle+\displaystyle \frac{1}{2}\langle \Sigma^t_i(x^{t+1}-x^t),x^{t+1}-x^t\rangle).
\]
Noticing that
\[
\begin{array}{ll}
[a_i]_+^2-[b_i]_+^2&=([a_i]_++[b_i]_+)([a_i]_+-[b_i]_+)\\[8pt]
&\leq (|a_i|+|b_i|)\cdot|a_i-b_i|\\[8pt]
&\leq (2|a_i|+|b_i-a_i|)\cdot|a_i-b_i|\\[8pt]
&=2|a_i|\cdot|a_i-b_i|+(a_i-b_i)^2\\[8pt]
&\leq 2|\lambda_i^t+\sigma g_i(x^t)|\cdot\sigma\left(\kappa_g\|x^{t+1}-x^t\|+\displaystyle \frac{1}{2}\|\Sigma^t_i\|
\|x^{t+1}-x^t\|^2\right)\\[8pt]
&\quad +\sigma^2\left(\kappa_g\|x^{t+1}-x^t\|+\displaystyle \frac{1}{2}\|\Sigma^t_i\|
\|x^{t+1}-x^t\|^2\right)^2,
\end{array}
\]
we obtain
\[
2\alpha\|x^{t+1}-x^t\|\leq 2\kappa_f+\sum_{i=1}^p\left[(2\kappa_g+\kappa_{H}D_0)|\lambda_i^t+\sigma g_i(x^t)|+\sigma(\kappa_g+\kappa_{H}D_0/2)^2\|x^{t+1}-x^t\|\right].
\]
If $2\alpha-p(\kappa_g+\kappa_{H}D_0/2)^2\sigma>0$, it yields
\[
\begin{array}{ll}
\|x^{t+1}-x^t\|&\leq \displaystyle \frac{2}{2\alpha-p(\kappa_g+\kappa_{H}D_0/2)^2\sigma}
\left(\kappa_f+(\kappa_g+\kappa_{H}D_0/2)\sum_{i=1}^p|\lambda_i^t+\sigma g_i(x^t)|\right) \\[12pt]
&\leq \displaystyle\frac{2}{2\alpha-p(\kappa_g+\kappa_{H}D_0/2)^2\sigma}
\left(\kappa_f+(\kappa_g+\kappa_{H}D_0/2)\sqrt{p}
[\|\lambda^t\|+\nu_g\sigma]\right),
\end{array}
\]
namely (\ref{eq:diffX}) is satisfied,
where the last inequality is obtained from the facts that $\sum_{i=1}^p|\lambda_i^t|\leq\sqrt{p}\|\lambda^t\|$ and $\sum_{i=1}^p|g_i(x^t)|\leq\sqrt{p}\|g(x^t)\|\leq\sqrt{p}\nu_g$.
\hfill $\Box$

\begin{lemma}\label{lem:1}
Let $(x^t,\lambda^t)$ be generated by QPALM,  Assumptions (A1)-- (A3), (B2) and  (B1) be satisfied.
Then
\begin{equation}\label{eq:lambda2}
\|\lambda^{t+1}\|^2 \leq \|\lambda^t\|^2 +2 \sigma \langle \lambda^t, g(x^{t+1})\rangle+\sigma^2 \nu_g^2,
\end{equation}
and
\begin{equation}\label{eq:2lambda}
\|\lambda^t\|- \gamma_1\sigma \leq \|\lambda^{t+1}\| \leq \|\lambda^t\|+ \gamma_1 \sigma.
\end{equation}
where
$$
\gamma_1=\nu_g+\sqrt{p} \left(\kappa_gD_0+\displaystyle \frac{1}{2}\kappa_{H}D_0^2\right).
$$
Moreover, one has
\begin{equation}\label{eq:lambdainorm}
|\lambda^{t+1}_i-\lambda^t_i| \leq \sigma \left[\nu_g+\left(\kappa_g+\displaystyle \frac{\kappa_{H}D_0}{2}\right)\|x^{t+1}-x^t\|\right]\leq \gamma_2\sigma,
\end{equation}
where
$$
\gamma_2=
\left[\nu_g+\left(\kappa_g D_0+\displaystyle \frac{\kappa_{H}D_0^2}{2}\right)\right].
$$
\end{lemma}
{\bf Proof}.
Noting that for any $a \in \mathbb R$, $[a]_+^2 \leq a^2$, we have from Assumptions (B2) and (A2) that
$$
\begin{array}{ll}
\|\lambda^{t+1}\|^2 &=\displaystyle \sum_{i=1}^p [\lambda_i^t+\sigma q^t_i(x^{t+1})]_+^2 \leq
\displaystyle \sum_{i=1}^p [\lambda^t_i+\sigma g_i(x^{t+1})]_+^2\leq \displaystyle \sum_{i=1}^p [\lambda^t_i+\sigma g_i(x^{t+1}))]^2\\[6pt]
&= \displaystyle \sum_{i=1}^p\left( [\lambda^t_i]^2+2\sigma g_i(x^{t+1}))+\sigma^2  g_i(x^{t+1})^2\right)\leq \|\lambda^t\|^2 +2 \sigma \langle \lambda^t, g(x^{t+1})\rangle+\sigma^2 \nu_g^2.
\end{array}
$$
It follows from the nonexpansion property of the projection $\Pi_{\mathbb R^p_+}(\cdot)$, we have from Assumptions (A2),(A3) and (B1)  that
$$
\begin{array}{ll}
\|\lambda^{t+1}-\lambda^t\|
& =\|[\lambda^t+\sigma q^t(x^{t+1})]_+-[\lambda^t]_+\|\\[6pt]
&  \leq \sigma\|g(x^t)\|+\sigma \left (\displaystyle \sum_{i=1}^p \left(\|\nabla_x g_i(x^t)\|\|x^{t+1}-x^t\|+\displaystyle \frac{1}{2}\|\Sigma^t_i\|\|x^{t+1}-x^t\|^2\right)^2\right)^{1/2}\\[8pt]
& \leq \sigma\nu_g+\sigma \left (\displaystyle \sum_{i=1}^p \left(\kappa_gD_0+\displaystyle \frac{1}{2}\kappa_{H}D_0^2\right)^2\right)^{1/2} \leq \sigma\left[\nu_g+\sqrt{p} \left(\kappa_gD_0+\displaystyle \frac{1}{2}\kappa_{H}D_0^2\right)\right],
\end{array}
$$
which implies (\ref{eq:2lambda}) from the definition of $\gamma_1$. Inequality (\ref{eq:lambdainorm}) can be obtained from the definition of $\lambda^{t+1}_i$ and Assumptions (A2), (A3) and  (B1). The proof is completed.
\hfill $\Box$\\

The following lemma is similar to Lemma 7 of \cite{Zhang2023}, however the assumptions here are different from those in  \cite{Zhang2023}.
\begin{lemma}\label{lemal7}
Let $\{Z(t)\}$ be a sequence  with $Z(0) = 0$. Suppose there exists an integer $t_0 >0$, real constants $\theta>0$, $\delta_{\max}>0$ and $ 0 <\zeta \leq \delta_{\max}$ such that
\begin{equation}\label{eq:Y9a}
|Z(t+1)-Z(t)| \leq \delta_{\max}
\end{equation}
and
\begin{equation}\label{eq:Y9}
\begin{array}{rl}
Z(t+t_0)-Z(t) & \leq \left
\{
\begin{array}{ll}
t_0 \delta_{\max} & \mbox{if } Z(t) < \theta\\[6pt]
-t_0\zeta & \mbox{if } Z(t) \geq \theta
\end{array}
\right.
\end{array}
\end{equation}
hold for all $t \in \{1,2,\ldots\}.$ Then the inequality holds
\begin{equation}\label{eq:re-ineq}
Z(t) \leq \theta +t_0 \delta_{\max}+t_0 \displaystyle \frac{4 \delta_{\max}^2}{\zeta}\log \left[ \displaystyle \frac{8 \delta_{\max}^2}{\zeta^2} \right], \forall t \in \{1,2,\ldots\}.
\end{equation}
\end{lemma}

\begin{lemma}\label{lemal5}
 Let $s > 0$ be an arbitrary integer. Let  Assumptions (A1) -- (A3) and (B3)-- (B4) be satisfied. Suppose $||\lambda^l|| \le M$ for all $l \in \{t, \ldots, t+s-1\}$. At each round $t \in \{1,2,\ldots\}$ in QPALM, for any $\alpha > 2\kappa_{\Sigma}(M)$ and
 \begin{equation}\label{eq:theta9}
 \vartheta (\sigma,\alpha,s, M)=\displaystyle \frac{\epsilon_0\sigma s}{2}+\gamma_1\sigma(s-1)+\displaystyle \frac{\alpha D_0^2}{\epsilon_0s}+\displaystyle \frac{\left(2\kappa_fD_0+\kappa_{\Sigma}(M)D_0^2\right)}{\epsilon_0}+
 \displaystyle \frac{\sigma \nu_g^2}{\epsilon_0},
 \end{equation}
 the following  holds
\begin{equation}\label{eq:6}
|\|\lambda^{t+1}\|-\|\lambda^t\||\leq \sigma \gamma_1
\end{equation}
and
\begin{equation}\label{eq:7}
 \|\lambda^{t+s}\|-\|\lambda^t\| \leq \left
\{
\begin{array}{ll}
s \sigma \gamma_1 & \mbox{if } \|\lambda^t\| < \vartheta (\sigma,\alpha, s, M),\\[6pt]
-s \displaystyle \frac{\sigma \epsilon_0}{2} & \mbox{if } \|\lambda^t\| \geq  \vartheta (\sigma,\alpha,s, M).
\end{array}
\right.
\end{equation}
\end{lemma}
{\bf Proof}. Inequality (\ref{eq:6}) follows from Lemma \ref{lem:1}. We only need to establish (\ref{eq:7}). Since it is obvious that
$$
 \|\lambda^{t+s}\|-\|\lambda^t\|\leq s \sigma \gamma_1
$$
 when $\|\lambda^t\|< \vartheta (\sigma,\alpha,s,M)$, it remains to prove
$$  \|\lambda^{t+s}\|-\|\lambda^t\| \leq -s \displaystyle \frac{\sigma \epsilon_0}{2}
$$
when $\|\lambda^t\| \geq  \vartheta (\sigma,\alpha,s, M)$.

For given positive integer $s$, suppose $\|\lambda^t\| \geq  \vartheta (\sigma,\alpha,s, M)$. For any $l \in \{t,t+1,\ldots, t+s-1\}$, since from Assumption (B4) that ${\cal L}^l_{\sigma }(x,\lambda^l) +\displaystyle \frac{\alpha}{2}\|x-x^l\|^2$ is strongly convex with modulus $\displaystyle \frac{\alpha}{2}$, one has from Assumption (B2) that
$$
\begin{array}{l}
\langle \nabla f(x^l),x^{l+1}-x^l \rangle+\displaystyle \frac{1}{2}\left \langle \Sigma^l_0(x^{l+1}-x^l), x^{l+1}-x^l \right \rangle+\displaystyle \frac{1}{2\sigma}\|\lambda^{l+1}\|^2+\displaystyle \frac{\alpha}{2}\|x^{l+1}-x^l\|^2\\[10pt]
\quad \,\leq \langle \nabla f(x^l)),\widehat x-x^l \rangle+\displaystyle \frac{1}{2}\left \langle \Sigma^l_0(\widehat x-x^l), \widehat x-x^l \right \rangle+\displaystyle \frac{1}{2\sigma}\|[\lambda^{l}+\sigma q^l(\widehat x)]_{+}\|^2\\[10pt]
\quad \quad \quad +\displaystyle \frac{\alpha}{2}\left[\|\widehat x-x^l\|^2-\|\widehat x-x^{l+1}\|^2\right]\\[10pt]
\quad \,\leq \langle \nabla f(x^l),\widehat x-x^l \rangle+\displaystyle \frac{1}{2}\left \langle \Sigma^l_0(\widehat x-x^l), \widehat x-x^l \right \rangle+\displaystyle \frac{1}{2\sigma}\|[\lambda^{l}+\sigma g(\widehat x)]_{+}\|^2\\[10pt]
\quad \quad \quad +\displaystyle \frac{\alpha}{2}\left[\|\widehat x-x^l\|^2-\|\widehat x-x^{l+1}\|^2\right].
\end{array}
$$
Using Assumption (A2) and the following inequality
$$
\|[\lambda^{l}+\sigma g(\widehat x)]_{+}\|^2 \leq \|\lambda^l\|^2+2\sigma \langle \lambda^l,g(\widehat x)\rangle
+\sigma^2\|g(\widehat x)\|^2,
$$
we obtain from Assumptions (A2),(A3),(B3) and (B1) that
\begin{equation}\label{eq:8}
\begin{array}{l}
\displaystyle \frac{1}{2\sigma} \left[\|\lambda^{l+1}\|^2-\|\lambda^l\|^2\right]
\leq \langle \nabla f(x^l)),\widehat x-x^{l+1}\rangle +\displaystyle \frac{1}{2}\left \langle \Sigma^l_0(\widehat x-x^l), \widehat x-x^l \right \rangle\\[10pt]
\quad \quad \quad +\displaystyle \frac{1}{2\sigma}\left[\|[\lambda^{l}+\sigma g(\widehat x)]_{+}\|^2-\|\lambda^l\|^2\right]-\left[\displaystyle \frac{\alpha}{2}\|x^{l+1}-x^l\|^2+\displaystyle \frac{1}{2}\left \langle \Sigma^l_0(x^{l+1}-x^l), x^{l+1}-x^l \right \rangle\right]\\[8pt]
\quad \quad \quad \quad \quad \quad +\displaystyle \frac{\alpha}{2}\left[\|\widehat x-x^l\|^2-\|\widehat x-x^{l+1}\|^2\right]\\[10pt]
\leq \kappa_fD_0+\displaystyle \frac{1}{2}\kappa_{\Sigma}(M)D_0^2+ \langle \lambda^l,g(\widehat x)\rangle
+\displaystyle \frac{\sigma}{2}\|g(\widehat x)\|^2 -\displaystyle \frac{\alpha}{2}\|x^{l+1}-x^l\|^2+\displaystyle \frac{\alpha}{2}\left[\|\widehat x-x^l\|^2-\|\widehat x-x^{l+1}\|^2\right]\\[10pt]
\leq \kappa_fD_0+\displaystyle \frac{1}{2}\kappa_{\Sigma}(M)D_0^2+ \langle \lambda^l,g(\widehat x)\rangle
+\displaystyle \frac{\sigma}{2}\nu_g^2-\displaystyle \frac{\alpha}{2}\|x^{l+1}-x^l\|^2+\displaystyle \frac{\alpha}{2}\left[\|\widehat x-x^l\|^2-\|\widehat x-x^{l+1}\|^2\right] \\[10pt]
\leq \kappa_fD_0+\displaystyle \frac{1}{2}\kappa_{\Sigma}(M)D_0^2+ \langle \lambda^l,g(\widehat x)\rangle
+\displaystyle \frac{\sigma}{2}\nu_g^2 +\displaystyle \frac{\alpha}{2}\left[\|\widehat x-x^l\|^2-\|\widehat x-x^{l+1}\|^2\right].
\end{array}
\end{equation}
Making  a summation of (\ref{eq:8}) over $\{t,t+1, \ldots, t+s-1\}$, we obtain  that
\begin{equation}\label{eq:9}
\begin{array}{l}
\displaystyle \frac{1}{2\sigma} [\|\lambda^{t+s}\|^2-\|\lambda^t\|^2 ]
\leq \left[\kappa_f D_0+\displaystyle \frac{1}{2}\kappa_{\Sigma}(M)D_0^2\right] s + \displaystyle \frac{\sigma}{2}\nu_g^2s +\displaystyle \sum_{l=t}^{t+s-1} \langle \lambda^l,g(\widehat x)\rangle
\\[10pt]
\quad \quad +\displaystyle \frac{\alpha}{2}\left(\|\widehat x-x^t\|^2-\|\widehat x-x^{t+s}\|^2\right)\\[10pt]
\leq \left[\kappa_f D_0+\displaystyle \frac{1}{2}\kappa_{\Sigma}(M)D_0^2\right]s + \displaystyle \frac{\sigma}{2}\nu_g^2s -\epsilon_0\displaystyle \sum_{l=0}^{s-1}\|\lambda^{t+l}\| +\displaystyle \frac{\alpha}{2}\left(\|\widehat x-x^t\|^2-\|\widehat x-x^{t+s}\|^2\right)\\[10pt]
\leq \left[\kappa_f D_0+\displaystyle \frac{1}{2}\kappa_{\Sigma}(M)D_0^2\right] s + \displaystyle \frac{\sigma}{2}\nu_g^2s -\epsilon_0\displaystyle \sum_{l=0}^{s-1} (\|\lambda^{t}\|-\sigma\gamma_1l)
\\[10pt]
\quad \quad +\displaystyle \frac{\alpha}{2}\left(\|\widehat x-x^t\|^2-\|\widehat x-x^{t+s}\|^2\right) \quad (\mbox{from } \|\lambda^{t+1}\|\geq \|\lambda^t\|-\sigma \gamma_1)\\[8pt]
\leq \left[\kappa_f D_0+\displaystyle \frac{1}{2}\kappa_{\Sigma}(M)D_0^2\right] s + \displaystyle \frac{\sigma}{2}\nu_g^2s +\displaystyle \frac{\alpha}{2}\left(\|\widehat x-x^t\|^2-\|\widehat x-x^{t+s}\|^2\right) \\[8pt]
\quad \quad + \epsilon_0\sigma \gamma_1 \displaystyle \frac{s(s-1)}{2}
-\epsilon_0\displaystyle \sum_{l=0}^{s-1} \|\lambda^{t}\|.
\end{array}
\end{equation}
From (\ref{eq:9}), we get that
$$
\begin{array}{l}
\|\lambda^{t+s}\|^2\,|\leq
\|\lambda^t\|^2 +2 \sigma \left[\kappa_f D_0+\displaystyle \frac{1}{2}\kappa_{\Sigma}(M)D_0^2\right] s+\sigma^2\nu_g^2s+\alpha \sigma D_0^2
+\epsilon_0\sigma^2\gamma_1s(s-1)-2\epsilon_0\sigma s \|\lambda^{t}\|\\[10pt]
=(\|\lambda^t\|-\displaystyle \frac{\epsilon_0\sigma}{2}s)^2
-\displaystyle \frac{\epsilon_0^2\sigma^2}{4}s^2+\epsilon_0\sigma^2 \gamma_1s(s-1)\\[10pt]
\quad \quad + \alpha \sigma D_0^2+2\sigma \left[\kappa_f D_0+\displaystyle \frac{1}{2}\kappa_{\Sigma}(M)D_0^2\right]s+\sigma^2\nu_g^2s-\epsilon_0\sigma s \|\lambda^{t}\| \\[10pt]
\leq  (\|\lambda^t\|-\displaystyle \frac{\epsilon_0\sigma}{2}s)^2
-\displaystyle \frac{3\epsilon_0^2\sigma^2}{4}s^2\\[10pt]
+\left( \displaystyle \frac{\epsilon_0^2\sigma^2}{2}s^2+\epsilon_0\sigma^2 \gamma_1s(s-1)+ \alpha \sigma D_0^2+2\sigma \left[\kappa_f D_0+\displaystyle \frac{1}{2}\kappa_{\Sigma}(M)D_0^2\right] s+\sigma^2\nu_g^2s-\epsilon_0\sigma s \vartheta (\sigma,\alpha,s,M)\right)\\[10pt]
= (\|\lambda^t\|-\displaystyle \frac{\epsilon_0\sigma}{2}s)^2-\displaystyle \frac{3\epsilon_0^2\sigma^2}{4}s^2\leq \left(\|\lambda^t\|-\displaystyle \frac{\epsilon_0\sigma}{2}s\right)^2.
\end{array}
$$
Taking square root on both sides yields the result. The proof is completed.
\hfill $\Box$

In order to use Lemma \ref{lemal7}  to analyze properties of QPALM for Problem (\ref{eq:1}), we introduce the following notations. Let  $\theta=\vartheta (\sigma,\alpha,s, M)$, $\delta_{\max}=\sigma \gamma_1$ and $\zeta =\displaystyle \frac{\sigma}{2}\epsilon_0$, and $t_0=s$, and define
$$
\psi(\sigma,\alpha,s, M)=\theta +t_0 \delta_{\max}+t_0 \displaystyle \frac{4 \delta_{\max}^2}{\zeta}\log \left[ \displaystyle \frac{8 \delta_{\max}^2}{\zeta^2} \right].
$$
 Then $\psi(\sigma,\alpha,s, M)$ is expressed as
$$
\begin{array}{ll}
\psi(\sigma,\alpha,s, M)= &
\vartheta (\sigma,\alpha,s, M)+  \left[\gamma_1+\displaystyle \frac{8\gamma_1^2}{\epsilon_0}\log\displaystyle \frac{32\gamma_1^2}{\epsilon_0^2}\right] \sigma s\\[12pt]
&=\kappa_0(M)+\kappa_1\displaystyle \frac{\alpha}{s}+ \kappa_2 s+\kappa_3
 \sigma+\kappa_4 \sigma s
 \end{array}
$$
  where
\begin{equation}\label{eq:notations}
\begin{array}{l}
\kappa_0(M)=\displaystyle \frac{\left(2\kappa_fD_0+\kappa_{\Sigma}(M)D_0^2\right)}{\epsilon_0},\,\,
\kappa_1=\displaystyle \frac{D_0^2}{\epsilon_0},\,\,
\kappa_2=0,\\[10pt]
\kappa_3=\displaystyle \frac{ \nu_g^2}{\epsilon_0}-\gamma_1,\,\,
\kappa_4=\left[\gamma_1+\displaystyle \frac{\epsilon_0}{2}+\displaystyle \frac{8\gamma_1^2}{\epsilon_0}\log \displaystyle \frac{32\gamma_1^2}{\epsilon_0^2}\right].
\end{array}
\end{equation}
Based on the above two lemmas, we obtain the following conclusion.

\begin{corollary}\label{cor:ebound}
 Let Assumptions (A1) -- (A3) and (B3) -- (B4) be satisfied and let $s>0$ be an arbitrary integer. Assume that $||\lambda^l|| \le M$ for all relevant prior steps. Then, it holds that
\begin{equation}\label{eq:bd1}
\|\lambda^k\| \leq \psi (\sigma,\alpha,s, M)
\end{equation}
If we choose $\alpha=\eta T^{1/3},\sigma=T^{-2/3},s=T^{1/3}$ for some constant $\eta>0$, then
\begin{equation}\label{eq:psiv}
\psi(\sigma,\alpha,s, M)=\psi(T^{-2/3},T^{1/3},T^{2/3}, M)=\kappa_0(M)+\kappa_1\displaystyle \frac{\alpha}{s}+ \kappa_2 s+\kappa_3 \sigma+\kappa_4 \sigma s=\kappa_0(M)+\eta\kappa_1+\kappa_3 T^{-2/3}+\kappa_4.
\end{equation}
In this case, we obtain for $k=1,\ldots, T$,
\begin{equation}\label{eq:lamB}
\|\lambda^k\| \leq \gamma_3(\eta, M),
\end{equation}
where $\gamma_3(\eta, M)=\kappa_0(M)+\eta\kappa_1+\kappa_3 +\kappa_4$.
\end{corollary}

\begin{lemma}\label{lemma:induction_bound}
Suppose Assumptions (A1)-(A5) and (B1)-(B5) hold. Let $\bar{\kappa}_0 = \frac{2\kappa_f D_0 + D_0^2}{\epsilon_0}$. We define a universal constant $M^* > 0$ independent of the total iteration $T$ as:
$$ M^* = \frac{\bar{\kappa}_0 + \eta \kappa_1 + 1}{1 - \rho}. $$
If the total number of iterations $T$ is chosen to be sufficiently large such that:
$$ \kappa_3 T^{-2/3} + \kappa_4 T^{-1/3} \le 1, $$
where $\kappa_3$ and $\kappa_4$ are the constants defined in (\ref{eq:notations}), then for all $1 \le t \le T$, the sequence generated by QPALM satisfies:
$$||\lambda^{t}|| \le M^* \quad \text{and} \quad ||\Sigma_{0}^{t}|| \le \kappa_{\Sigma}(M^*).$$
\end{lemma}

{\bf Proof}.
We prove this by mathematical induction on $t$. Recall from (\ref{eq:notations}) and (\ref{eq:psiv}) that the theoretical drift bound $\psi$ over $s$ steps explicitly depends on $M$ through $\kappa_{\Sigma}(M)$. By separating the linear growth term, we have for any $M$:
\begin{align*}
\psi(M, T) &= \kappa_0(M) + \eta \kappa_1 + \kappa_3 T^{-2/3} + \kappa_4 T^{-1/3} \\
&= \frac{2\kappa_f D_0 + (1 + C_{\Sigma} M)D_0^2}{\epsilon_0} + \eta \kappa_1 + \kappa_3 T^{-2/3} + \kappa_4 T^{-1/3} \\
&= \bar{\kappa}_0 + \rho M + \eta \kappa_1 + \kappa_3 T^{-2/3} + \kappa_4 T^{-1/3}.
\end{align*}
We prove the result by induction. When $t=1$, since $\lambda^1 = 0$, we trivially have $\|\lambda^1\| = 0 \le M^*$. Consequently, $\Sigma_0^1 = I$, yielding $\|\Sigma_0^1\| = 1 \le \kappa_{\Sigma}(M^*)$.
 Assume that for a given step $t$ ($1 \le t < T$), the condition holds for all $l \in \{1, 2, \dots, t\}$: $\|\lambda^l\| \le M^*$ and $\|\Sigma_0^l\| \le \kappa_{\Sigma}(M^*)$. We need to show that $\lambda^{t+1}\|\leq M^*$. Since the uniform bound $M^*$ holds for all previous steps up to $t$, the strong convexity and the gradient Lipschitz conditions parametrized by $\kappa_{\Sigma}(M^*)$ remain valid. We can seamlessly apply the sequence drift bound from Lemma \ref{lemal7} along with Lemma \ref{lemal5} to obtain the bound for the next iteration:
$$ \|\lambda^{t+1}\| \le \psi(M^*, T). $$
Substituting the definition of $M^*$ and utilizing the assumption on sufficiently large $T$, we have:
\begin{align*}
\|\lambda^{t+1}\| &\le \bar{\kappa}_0 + \rho M^* + \eta \kappa_1 + \left( \kappa_3 T^{-2/3} + \kappa_4 T^{-1/3} \right) \\
&\le \bar{\kappa}_0 + \rho M^* + \eta \kappa_1 + 1 \\
&= \bar{\kappa}_0 + \eta \kappa_1 + 1 + \rho \frac{\bar{\kappa}_0 + \eta \kappa_1 + 1}{1 - \rho} \\
&= \frac{(\bar{\kappa}_0 + \eta \kappa_1 + 1)(1 - \rho) + \rho(\bar{\kappa}_0 + \eta \kappa_1 + 1)}{1 - \rho} \\
&= \frac{\bar{\kappa}_0 + \eta \kappa_1 + 1}{1 - \rho} = M^*.
\end{align*}
Therefore, $\|\lambda^{t+1}\| \le M^*$ holds. It directly follows that $\|\Sigma_0^{t+1}\| \le 1 + C_{\Sigma} M^* = \kappa_{\Sigma}(M^*)$.
By the principle of mathematical induction, the uniform bound holds for all $t \in \{1, \dots, T\}$, breaking the circular dependency between $\|\Sigma_0^t\|$ and $\|\lambda^t\|$.
\hfill $\Box$

\section{Iteration Complexities of QPALM}\label{Sec3}
\setcounter{equation}{0}
Based on Lemma \ref{lemma:induction_bound}, for a sufficiently large $T$, the multiplier sequence generated by QPALM is uniformly bounded by $M^*$ for all $1 \le t \le T$. Consequently, the matrix sequence $\Sigma_0^t$ is universally bounded by $\kappa_\Sigma(M^*)$. In the following non-asymptotic analysis, to maintain consistent and simplified notations, we will drop the explicit $M^*$ dependency and define the fixed constants: $\kappa_\Sigma := \kappa_\Sigma(M^*)$, $\gamma_1 := \gamma_1$, $\gamma_2 := \gamma_2$, $\gamma_3 := \gamma_3(\eta, M^*)$, $\gamma_4 := \gamma_4(M^*)$, $\gamma_5 := \gamma_5(M^*)$, $\gamma_6 := \gamma_6(M^*)$, and $\widehat\beta_k(\sigma) := \widehat\beta_k(\sigma, M^*)$. This convention effectively decouples the parameter interdependency and validates the strong convexity and Lipschitz constants used in the subsequent non-asymptotic analysis.

Since Problem (\ref{eq:1}) is nonconvex, we can only discuss how the sequence generated by QPALM satisfies the Karush-Kuhn-Tucker (KKT) conditions of Problem (\ref{eq:1}). We will establish the complexities of QPALM  for finding an $\varepsilon$-approximate KKT point of Problem (\ref{eq:1}). Define
\begin{equation}\label{eq:bk}
\beta_k( \sigma)=L_0+\left(\displaystyle \sum_{j=1}^pL_j\gamma_2\right)k \sigma.
\end{equation}

\begin{lemma}\label{eq:bka}
Let Assumptions (A1)--(A3),(A5), (B1) and (B2) be satisfied. Then $x \rightarrow L(x,\lambda^k)$ is $\beta_k(\sigma)$-weakly convex.
\end{lemma}
{\bf Proof}.
Since (A5) holds,  $\lambda\geq 0$,
 $$ x \rightarrow L(x,\lambda^k)=f(x)+\displaystyle \sum_{j=1}^p \lambda^k_jg_j(x)$$
 is $L_0+\displaystyle \sum_{j=1}^p \lambda^k_jL_j$-weakly convex. In view of (\ref{eq:lambdainorm}) and $\lambda^1=0$, we have $\lambda^k_j \leq \gamma_2 k \sigma$. This observation yields the result.
\hfill $\Box$

 \begin{lemma}\label{lemaugP}
 For $\lambda\geq 0$ and $\sigma >0$, one has that $x \rightarrow {\cal L}^k_{\sigma}(x,\lambda)$ is positive definite  if $\Sigma_0^k+\displaystyle \sum_{j=1}^p\lambda_j \Sigma^k_j$ is positively (semi-definite) definite.
 \end{lemma}
 {\bf Proof}.
 Noting that
 $$
 {\cal L}^k_{\sigma}(x,\lambda)=\min_{y \leq 0} \widehat \phi (x,y,\lambda)=q^k_0(x)+\langle \lambda, q^k(x)-y\rangle
 +\displaystyle \frac{\sigma}{2}\|q^k(x)-y\|^2,
 $$
 we only need to prove $(x,y) \rightarrow \widehat \phi (x,y,\lambda)$ is a convex function. It is easy to obtain
 $$
 \nabla_{x,y}^2 \widehat \phi (x,y,\lambda)=\left[
 \begin{array}{cc}
 \Sigma^k_0+\displaystyle \sum_{j=1}^p \lambda_j \Sigma^k_j+\sigma {\cal J}q^k(x)^T{\cal J}q^k(x) & -\sigma {\cal J}q^k(x)^T\\[6pt]
 -\sigma {\cal J} q^k(x) & \sigma I
 \end{array}
 \right].
 $$
 Since the Schur complement matrix of $\sigma I$ in $\nabla_{x,y}^2 \widehat \phi (x,y,\lambda)$ is
 $$
 \nabla_{x,y}^2 \widehat \phi (x,y,\lambda)/(\sigma I)=\Sigma^k_0+\displaystyle \sum_{j=1}^p \lambda_j \Sigma^k_j,
 $$
 we obtain that $\nabla_{x,y}^2 \widehat \phi (x,y,\lambda)$ is positively (semi-definite) definite if and only if $\Sigma^k_0+\displaystyle \sum_{j=1}^p \lambda_j \Sigma^k_j$ is positively (semi-definite) definite.
 \hfill $\Box$

 \begin{proposition}\label{corPosiSig}
 If (A1)-(A3), (A5) hold, $\Sigma^k_i \prec -L_i I$ for $i=1,\ldots p$, and $\Sigma^k_0 = -\sum_{j=1}^p \lambda_j^k \Sigma_j^k + I \succ 0$, then Assumptions (B1) -- (B4) hold.
 \end{proposition}

\begin{lemma}\label{lempls}
For any vectors $a,b \in \mathbb R^p$, one has
\begin{equation}\label{plsexp}
\|[b]_+\|^2-\|[a]_+\|^2\leq 2 \langle [a]_+, b-a\rangle+\|b-a\|^2.
\end{equation}
\end{lemma}
{\bf Proof}.
The result comes from the following inequality
$$
[b_i]_+^2-[a_i]_+^2\leq 2 \langle [a_i]_+, b_i-a_i\rangle+|b_i-a_i|^2, \, i=1 \ldots, p.
$$
The proof is completed.
\hfill $\Box$\\
Define
$$
\Delta^k(z)=\left[
\begin{array}{c}
\nabla g_1(x^k)^T(z-x^k)+\displaystyle \frac{1}{2}(z-x^k)^T\Sigma^1_i(z-x^k)\\[2pt]
\vdots\\[2pt]
\nabla g_p(x^k)^T(z-x^k)+\displaystyle \frac{1}{2}(z-x^k)^T\Sigma^p_i(z-x^k)
\end{array}
\right]\mbox{  and  }\Delta_{\Sigma}^k(z)=\left[
\begin{array}{c}
\displaystyle \frac{1}{2}(z-x^k)^T\Sigma^1_i(z-x^k)\\[2pt]
\vdots\\[2pt]
\displaystyle \frac{1}{2}(z-x^k)^T\Sigma^p_i(z-x^k)
\end{array}
\right].
$$
From Lemma \ref{lempls}, we have
\begin{equation}\label{eqnormminus}
\begin{array}{l}
\|[\lambda^k+\sigma q^k(x)]_+\|^2-\|[\lambda^k+\sigma g(x^k)]_+\|^2\\[6pt]
\quad \quad \leq 2\sigma[\lambda^k+\sigma g(x^k)]_+^T({\cal J}g(x^k)(x-x^k)+\Delta_{\Sigma}^k(z))+\sigma^2
\|\Delta^k(z))\|^2.
\end{array}
\end{equation}
Define
\begin{equation}\label{eqga4}
\gamma_4=
\left\{\nu_g\gamma_1
+
\gamma_1[p^{1/2}\kappa_gD_0+p^{1/2}D_0^2\kappa_{H}]+\displaystyle \frac{1}{2}[p^{1/2}\kappa_gD_0+p^{1/2}D_0^2\kappa_{H}]^2\right\}.
\end{equation}
and
\begin{equation}\label{betahatdef}
\widehat \beta_k(\sigma)=\beta_k(\sigma)+\displaystyle\frac{1}{2}\kappa_{\Sigma}+\displaystyle \frac{\sigma}{2}\left[2\nu_gp^{1/2}\kappa_{H}+2p\kappa_g^2+\displaystyle \frac{1}{2}p\kappa_{H}^2D_0^2  \right]+p^{1/2}\kappa_{H}\gamma_1k\sigma.
\end{equation}

\begin{theorem}\label{th:L-est}
Let Assumptions (A1) -- (A3), (A5) and (B1) -- (B4) be satisfied. Suppose
\begin{equation}\label{eqal1}
\alpha> 16 \hat \beta_k(\sigma),
\end{equation}
Then
\begin{equation}\label{eqfinalesti}
\begin{array}{l}
 \|\nabla e_{\alpha^{-1}\psi^k}(x^k)\|^2
\leq 16 \alpha \left[ e_{\alpha^{-1}\psi^k}(x^k)-e_{\alpha^{-1}\psi^{k+1}}(x^{k+1})+
\gamma_4 \sigma \right].
\end{array}
\end{equation}
\end{theorem}
{\bf Proof}.
In view of (\ref{eq:opt-x-1}), we have from (\ref{eqnormminus}) that
\begin{equation}\label{eq:opt-x-1a}
\begin{array}{ll}
\displaystyle\langle \nabla_x f(x^k),x^{k+1}-x^k \rangle  +\displaystyle \frac{1}{2}\langle \Sigma^t_0(x^{t+1}-x^t),x^{t+1}-x^t\rangle+ \frac{1}{2\sigma}\|\lambda^{k+1}\|^2
+ \frac{\alpha}{2} \|x^{k+1}-x^k\|^2  \\[5pt]
\leq \displaystyle\langle \nabla_x f(x^k),z-x^k \rangle +\displaystyle \frac{1}{2}\langle \Sigma^t_0(z-x^t),z-x^t\rangle \\[5pt]
\quad\quad +
\displaystyle \frac{1}{2\sigma}\|[\lambda^k+\sigma g(x^k)]_+\|^2+
[\lambda^k+\sigma g(x^k)]_+^T{\cal J}g(x^k)(z-x^k)\\[5pt]
\quad\quad+[\lambda^k+\sigma g(x^k)]_+^T\Delta^k_{\Sigma}(z) +\displaystyle \frac{\sigma}{2}
\|\Delta^k(z))\|^2+\displaystyle \frac{\alpha}{2}(\|z-x^k\|^2-\|z-x^{k+1}\|^2)\\[5pt]
\leq \displaystyle\langle \nabla_x L(x^k,\lambda^k),z-x^k \rangle +\displaystyle \frac{1}{2}\langle \Sigma^t_0(z-x^t),z-x^t\rangle \\[5pt]
\quad\quad +
\displaystyle \frac{1}{2\sigma}\|[\lambda^k+\sigma g(x^k)]_+\|^2+
([\lambda^k+\sigma g(x^k)]_+-\lambda^k)^T{\cal J}g(x^k)(z-x^k)\\[5pt]
\quad\quad+[\lambda^k+\sigma g(x^k)]_+^T\Delta^k_{\Sigma}(z) +\displaystyle \frac{\sigma}{2}
\|\Delta^k(z))\|^2+\displaystyle \frac{\alpha}{2}(\|z-x^k\|^2-\|z-x^{k+1}\|^2)
\end{array}
\end{equation}
Since $x \rightarrow L(x,\lambda^k)$ is $\beta_k(\sigma)$-weakly convex, we have
$$
L(z,\lambda^k)-L(x^k,\lambda^k)\geq \displaystyle\langle \nabla_x L(x^k,\lambda^k),z-x^k \rangle
-\frac{\beta_k(\sigma)}{2} \|z-x^k\|^2,
$$
and from this we obtain
\begin{equation}\label{eq:opt-x-1aa}
\begin{array}{ll}
\displaystyle\langle \nabla_x f(x^k),x^{k+1}-x^k \rangle  +\displaystyle \frac{1}{2}\langle \Sigma^t_0(x^{t+1}-x^t),x^{t+1}-x^t\rangle+ \frac{1}{2\sigma}\|\lambda^{k+1}\|^2
+ \frac{\alpha}{2} \|x^{k+1}-x^k\|^2  \\[5pt]
\leq L(x,\lambda^k)-L(x^k,\lambda^k)+{\frac{\beta_k(\sigma)}{2}} \|z-x^k\|^2+\displaystyle \frac{1}{2}\langle \Sigma^t_0(z-x^t),z-x^t\rangle \\[5pt]
\quad\quad +
\displaystyle \frac{1}{2\sigma}\|[\lambda^k+\sigma g(x^k)]_+\|^2+
\nu_gp^{1/2}\kappa_g\sigma\|z-x^k\|+\displaystyle \frac{\alpha}{2}(\|z-x^k\|^2-\|z-x^{k+1}\|^2)\\[5pt]
\quad\quad+ {\frac{1}{2}}\|\lambda^k+\sigma g(x^k)\|p^{1/2}\kappa_{H}\|z-x^k\|^2 +\displaystyle \frac{\sigma}{2}\left[2p\kappa_g^2+\displaystyle \frac{1}{2}p\kappa_{H}^2D_0^2  \right]\|z-x^k\|^2\end{array}
\end{equation}
\[
\begin{array}{l}
\leq
L(x,\lambda^k)-L(x^k,\lambda^k)+{\frac{\beta_k(\sigma)}{2}} \|z-x^k\|^2+\displaystyle \frac{1}{2}\langle \Sigma^t_0(z-x^t),z-x^t\rangle \\[5pt]
\quad\quad +
\displaystyle \frac{1}{2\sigma}\|\lambda^{k+1}\|^2-\langle \lambda^{k+1},\Delta^k(x^{k+1})\rangle
+\displaystyle \frac{\sigma}{2}\|\Delta^k(x^{k+1})\|^2+{\nu_gp^{1/2}\kappa_g\sigma\|z-x^k\|}\\[5pt]
\quad \quad +\displaystyle \frac{\alpha}{2}(\|z-x^k\|^2-\|z-x^{k+1}\|^2)\\[5pt]
\quad\quad+{\frac{1}{2}}\|\lambda^k+\sigma g(x^k)\|p^{1/2}\kappa_{H}\|z-x^k\|^2 +\displaystyle \frac{\sigma}{2}\left[2p\kappa_g^2+\displaystyle \frac{1}{2}p\kappa_{H}^2D_0^2  \right]\|z-x^k\|^2
\end{array}
\]
From the $\beta_{k+1}$-weakly convexity of $L(x,\lambda^{k+1})$, we obtain
\begin{equation}
    \begin{aligned}
        & - \langle \nabla_x f(x^k),x^{k+1}-x^k \rangle - L(x^k,y^k) - \langle \lambda^{k+1}, \Delta^{k}(x^{k+1}) \rangle \\
         \leq &- L(x^{k+1},y^k)  + \frac{\beta_{k}(\sigma)}{2} \Vert x^{k+1} - x^{k} \Vert^2 -\langle \lambda^k , \Delta_{\Sigma}^{k}(x^{k+1}) \rangle + \langle \lambda^k - \lambda^{k+1}, \Delta^{k}(x^{k+1}) \rangle
    \end{aligned}
\end{equation}
and
\begin{equation}\label{eq:opt-x-1ab}
\begin{array}{ll}
0\leq -\displaystyle\langle \nabla_x f(x^k),x^{k+1}-x^k \rangle -\displaystyle \frac{1}{2}\langle \Sigma^t_0(x^{t+1}-x^t),x^{t+1}-x^t\rangle- \frac{1}{2\sigma}\|\lambda^{k+1}\|^2
- \frac{\alpha}{2} \|x^{k+1}-x^k\|^2  \\[5pt]
\quad \quad +L(x,\lambda^k)-L(x^k,\lambda^k)+\beta_k(\sigma) \|z-x^k\|^2+\displaystyle \frac{1}{2}\langle
\Sigma^t_0(z-x^t),z-x^t\rangle \\[5pt]
\quad\quad +
\displaystyle \frac{1}{2\sigma}\|[\lambda^k+\sigma g(x^k)]_+\|^2+
\nu_gp^{1/2}\kappa_g\sigma\|z-x^k\|+\displaystyle \frac{\alpha}{2}(\|z-x^k\|^2-\|z-x^{k+1}\|^2)\\[5pt]
\quad\quad+\|\lambda^k+\sigma g(x^k)\|p^{1/2}\kappa_{H}\|z-x^k\|^2 +\displaystyle \frac{\sigma}{2}\left[2p\kappa_g^2+\displaystyle \frac{1}{2}p\kappa_{H}^2D_0^2  \right]\|z-x^k\|^2\\[5pt]
\leq -\displaystyle \frac{1}{2}\langle \Sigma^t_0(x^{t+1}-x^t),x^{t+1}-x^t\rangle
- \frac{\alpha}{2} \|x^{k+1}-x^k\|^2+\displaystyle \frac{\alpha}{2}(\|z-x^k\|^2-\|z-x^{k+1}\|^2)  \\[5pt]
\quad \quad +L(x,\lambda^k)-L(x^k,\lambda^k)+\beta_k(\sigma) \|z-x^k\|^2+\displaystyle \frac{1}{2}\langle \Sigma^t_0(z-x^t),z-x^t\rangle \\[5pt]
\quad\quad -(\displaystyle\langle \nabla_x f(x^k),x^{k+1}-x^k \rangle +\langle \lambda^{k},\Delta^k(x^{k+1})\rangle)+\langle \lambda^k-\lambda^{k+1},\Delta^k(x^{k+1})\rangle
+\displaystyle \frac{\sigma}{2}\|\Delta^k(x^{k+1})\|^2\\[5pt]
\quad\quad+\|\lambda^k+\sigma g(x^k)\|p^{1/2}\kappa_{H}\|z-x^k\|^2 +\displaystyle \frac{\sigma}{2}\left[2p\kappa_g^2+\displaystyle \frac{1}{2}p\kappa_{H}^2D_0^2  \right]\|z-x^k\|^2
\\[5pt]
=-\displaystyle \frac{1}{2}\langle \Sigma^t_0(x^{t+1}-x^t),x^{t+1}-x^t\rangle
- \frac{\alpha}{2} \|x^{k+1}-x^k\|^2 +\displaystyle \frac{\alpha}{2}(\|z-x^k\|^2-\|z-x^{k+1}\|^2) \\[5pt]
\quad \quad +L(x,\lambda^k)-L(x^k,\lambda^k)+\beta_k(\sigma) \|z-x^k\|^2+\displaystyle \frac{1}{2}\langle \Sigma^t_0(z-x^t),z-x^t\rangle \\[5pt]
\quad\quad -\displaystyle\langle \nabla_x L(x^k,\lambda^k),x^{k+1}-x^k \rangle -\langle \lambda^{k},\Delta^k_{\Sigma}(x^{k+1})\rangle)+\langle \lambda^k-\lambda^{k+1},\Delta^k(x^{k+1})\rangle
+\displaystyle \frac{\sigma}{2}\|\Delta^k(x^{k+1})\|^2\\[5pt]
\quad\quad+\|\lambda^k+\sigma g(x^k)\|p^{1/2}\kappa_{H}\|z-x^k\|^2 +\displaystyle \frac{\sigma}{2}\left[2p\kappa_g^2+\displaystyle \frac{1}{2}p\kappa_{H}^2D_0^2  \right]\|z-x^k\|^2
\\[5pt]
\leq -\displaystyle \frac{1}{2}\langle \Sigma^t_0(x^{t+1}-x^t),x^{t+1}-x^t\rangle
- \frac{\alpha-\beta_k(\sigma)}{2} \|x^{k+1}-x^k\|^2 +\displaystyle \frac{\alpha}{2}(\|z-x^k\|^2-\|z-x^{k+1}\|^2) \\[5pt]
\quad \quad +L(x,\lambda^k)-L(x^{k+1},\lambda^k)+\beta_k(\sigma) \|z-x^k\|^2+\displaystyle \frac{1}{2}\langle \Sigma^t_0(z-x^t),z-x^t\rangle \\[5pt]
\quad\quad -\langle \lambda^{k},\Delta^k_{\Sigma}(x^{k+1})\rangle)+\langle \lambda^k-\lambda^{k+1},\Delta^k(x^{k+1})\rangle
+\displaystyle \frac{\sigma}{2}\|\Delta^k(x^{k+1})\|^2\\[5pt]
\quad\quad+\|\lambda^k+\sigma g(x^k)\|p^{1/2}\kappa_{H}\|z-x^k\|^2 +\displaystyle \frac{\sigma}{2}\left[2p\kappa_g^2+\displaystyle \frac{1}{2}p\kappa_{H}^2D_0^2  \right]\|z-x^k\|^2
.
\end{array}
\end{equation}
From the definition of $\widehat \beta_k(\sigma)$, we have
\begin{equation}\label{betahat}
\widehat \beta_k(\sigma)\geq {\frac{1}{2}}\beta_k(\sigma)+\displaystyle\frac{1}{2}\|\Sigma^t_0\|+ {\frac{1}{2}}\|\lambda^k+\sigma g(x^k)\|p^{1/2}\kappa_{H}+\displaystyle \frac{\sigma}{2}\left[2p\kappa_g^2+\displaystyle \frac{1}{2}p\kappa_{H}^2D_0^2  \right].
\end{equation}
Then we have a simple inequality from (\ref{eq:opt-x-1ab}) as follows
\begin{equation}\label{eq:opt-x-1ac}
\begin{array}{ll}
0\leq  -\displaystyle \frac{1}{2}\langle \Sigma^t_0(x^{t+1}-x^t),x^{t+1}-x^t\rangle
- \frac{\alpha-\beta_k(\sigma)}{2} \|x^{k+1}-x^k\|^2+\displaystyle \frac{\alpha}{2}(\|z-x^k\|^2-\|z-x^{k+1}\|^2)   \\[8pt]
\quad\quad +L(x,\lambda^k)-L(x^{k+1},\lambda^k)+\widehat \beta_k(\sigma) \|z-x^k\|^2 +{\nu_gp^{1/2}\kappa_g\sigma\|z-x^k\|}\\[5pt]
\quad\quad -\langle \lambda^{k},\Delta^k_{\Sigma}(x^{k+1})\rangle)+\langle \lambda^k-\lambda^{k+1},\Delta^k(x^{k+1})\rangle
+\displaystyle \frac{\sigma}{2}\|\Delta^k(x^{k+1})\|^2.
\end{array}
\end{equation}
Let
$$
\psi_k(z)=L(z,\lambda^k)+\delta_X(x)
$$
and
$$
\hat x^k={\rm prox}_{\alpha^{-1}\psi^k}(x^k),
$$
then from  \cite[Theorem 2.26]{RW98}, we obtain
\begin{equation}\label{envg}
\hat x^k-x^k=\alpha^{-1}\nabla e_{\alpha^{-1}\psi^k}(x^k).
\end{equation}
For $x^{k+1}$, we have
\begin{equation}\label{envpl}
e_{\alpha^{-1}\psi^k}(x^{k+1})=\inf_{z \in X}\left\{L(z,\lambda^k)+\displaystyle
\frac{\alpha}{2}\|z-x^{k+1}\|^2\right\}\leq L(x^{k+1}, \lambda^k).
\end{equation}
Setting $z=\hat x^k$ in (\ref{eq:opt-x-1ac}), we get
\begin{equation}\label{eq:opt-x-1ad}
\begin{array}{ll}
0\leq  -\displaystyle \frac{1}{2}\langle \Sigma^t_0(x^{t+1}-x^t),x^{t+1}-x^t\rangle
- \frac{\alpha-\beta_k(\sigma)}{2} \|x^{k+1}-x^k\|^2  \\[5pt]
\quad \quad + e_{\alpha^{-1}\psi^k}(x^k)-L(x^{k+1},\lambda^k)+\widehat \beta_k(\sigma) \|\hat x^k-x^k\|^2 -\displaystyle \frac{\alpha}{2}\|\hat x^k-x^{k+1}\|^2\\[15pt]
\quad\quad -\langle \lambda^{k},\Delta^k_{\Sigma}(x^{k+1})\rangle)+\langle \lambda^k-\lambda^{k+1},\Delta^k(x^{k+1})\rangle
+\displaystyle \frac{\sigma}{2}\|\Delta^k(x^{k+1})\|^2 +{\nu_gp^{1/2}\kappa_g\sigma\|\hat{x}-x^k\|}\\[15pt]
\leq  -\displaystyle \frac{1}{2}\langle \Sigma^t_0(x^{t+1}-x^t),x^{t+1}-x^t\rangle
- \frac{\alpha-\beta_k(\sigma)}{2} \|x^{k+1}-x^k\|^2  \\[5pt]
\quad \quad + e_{\alpha^{-1}\psi^k}(x^k)-e_{\alpha^{-1}\psi^k}(x^{k+1})+\widehat \beta_k(\sigma) \|\hat x^k-x^k\|^2-\displaystyle \frac{\alpha}{2}\|\hat x^k-x^{k+1}\|^2 \\[15pt]
\quad\quad -\langle \lambda^{k},\Delta^k_{\Sigma}(x^{k+1})\rangle)+\langle \lambda^k-\lambda^{k+1},\Delta^k(x^{k+1})\rangle
+\displaystyle \frac{\sigma}{2}\|\Delta^k(x^{k+1})\|^2 +{\nu_gp^{1/2}\kappa_g\sigma\|\hat{x}-x^k\|}.
\end{array}
\end{equation}
Noting for $w^{k+1}={\rm prox}_{\alpha^{-1}\psi_k}(x^{k+1})$, we have
$$
\begin{array}{l}
e_{\alpha^{-1}\psi_{k+1}}(x^{k+1})-e_{\alpha^{-1}\psi^k}(x^{k+1})\\[4pt]
\leq L(w^{k+1},\lambda^{k+1})+\displaystyle \frac{\alpha}{2}
\|w^{k+1}-x^{k+1}\|^2-\left(L(w^{k+1},\lambda^{k})+\displaystyle \frac{\alpha}{2}
\|w^{k+1}-x^{k+1}\|^2  \right)\\[10pt]
=(\lambda^{k+1}-\lambda^k)^Tg(w^{k+1})
\leq \nu_g\|\lambda^{k+1}-\lambda^k\|\leq \nu_g\gamma_1\sigma \quad \text{(from Lemma 2.3)}.
\end{array}
$$
Thus, we obtain from (\ref{eq:opt-x-1ad}) that
\begin{equation}\label{eq:opt-x-1ae}
\begin{array}{ll}
0
\leq  -\displaystyle \frac{1}{2}\langle \Sigma^t_0(x^{t+1}-x^t),x^{t+1}-x^t\rangle
- \frac{\alpha-\beta_k(\sigma)}{2} \|x^{k+1}-x^k\|^2 -\displaystyle \frac{\alpha}{2}\|\hat x^k-x^{k+1}\|^2 \\[8pt]
\quad \quad + e_{\alpha^{-1}\psi^k}(x^k)-e_{\alpha^{-1}\psi^{k+1}}(x^{k+1})+{\nu_g[\gamma_1+p^{1/2}\kappa_{g}D_{0}]\sigma}+\widehat \beta_k(\sigma) \|\hat x^k-x^k\|^2 \\[8pt]
\quad\quad -\langle \lambda^{k},\Delta^k_{\Sigma}(x^{k+1})\rangle)+\langle \lambda^k-\lambda^{k+1},\Delta^k(x^{k+1})\rangle
+\displaystyle \frac{\sigma}{2}\|\Delta^k(x^{k+1})\|^2.
\end{array}
\end{equation}
Using a simple inequality $(a-b)^2\geq \displaystyle \frac{1}{4}a^2-\displaystyle \frac{1}{3}b^2$, we have
$$
-\displaystyle \frac{\alpha}{2}\|\hat x^k-x^{k+1}\|^2\leq- \displaystyle \frac{\alpha}{8}\|\hat x^k-x^k\|^2+\displaystyle \frac{\alpha}{6}\|x^{k+1}-x^k\|^2.
$$
Then we obtain from (\ref{eq:opt-x-1ae}) that
\begin{equation}\label{eq:opt-x-1af}
\begin{array}{ll}
\left(\displaystyle \frac{\alpha}{8}-\widehat \beta_k(\sigma)\right)\|\hat x^k-x^k\|^2
\leq  -\displaystyle \frac{1}{2}\langle \Sigma^t_0(x^{t+1}-x^t),x^{t+1}-x^t\rangle
- \left(\frac{\alpha}{3} -\frac{\beta_k(\sigma)}{2}\right)\|x^{k+1}-x^k\|^2  \\[10pt]
\quad \quad \quad + e_{\alpha^{-1}\psi^k}(x^k)-e_{\alpha^{-1}\psi^{k+1}}(x^{k+1})+{\nu_g[\gamma_1+p^{1/2}\kappa_{g}D_{0}]\sigma} \\[10pt]
\quad\quad \quad -\langle \lambda^{k},\Delta^k_{\Sigma}(x^{k+1})\rangle)+\langle \lambda^k-\lambda^{k+1},\Delta^k(x^{k+1})\rangle
+\displaystyle \frac{\sigma}{2}\|\Delta^k(x^{k+1})\|^2\\[10pt]
\leq -\displaystyle \frac{1}{2}\left\langle \left(\Sigma^k_0-\displaystyle \sum_{i=1}^p \lambda^k_i \Sigma^k_i\right)(x^{t+1}-x^t),x^{t+1}-x^t\right\rangle
- \left(\frac{\alpha}{3} -\frac{\beta_k(\sigma)}{2}\right)\|x^{k+1}-x^k\|^2  \\[15pt]
\quad \quad  + e_{\alpha^{-1}\psi^k}(x^k)-e_{\alpha^{-1}\psi^{k+1}}(x^{k+1})\\[8pt]
\quad \quad +\left\{{\nu_g[\gamma_1 + p^{1/2}\kappa_{g}D_{0}]}
+
\gamma_1[p^{1/2}\kappa_gD_0+p^{1/2}D_0^2\kappa_{H}]+\displaystyle \frac{1}{2}[p^{1/2}\kappa_gD_0+p^{1/2}D_0^2\kappa_{H}]^2\right\}
\sigma\\[10pt]
\leq -\displaystyle \frac{1}{2}\left\langle \left(\Sigma^k_0-\displaystyle \sum_{i=1}^p \lambda^k_i \Sigma^k_i\right)(x^{t+1}-x^t),x^{t+1}-x^t\right\rangle
- \left(\frac{\alpha}{3} -\frac{\beta_k(\sigma)}{2}\right)\|x^{k+1}-x^k\|^2  \\[15pt]
\quad \quad  + e_{\alpha^{-1}\psi^k}(x^k)-e_{\alpha^{-1}\psi^{k+1}}(x^{k+1})+
\gamma_4 \sigma
\end{array}
\end{equation}
where $\gamma_4$ is defined by (\ref{eqga4}).
Suppose
\begin{equation}\label{eqal1a}
\alpha> \max \{ 16 \hat \beta_k(\sigma), 3\beta_k(\sigma)/2\}=16 \hat \beta_k(\sigma),
\end{equation}
then (\ref{eq:opt-x-1af}) implies
$$
\begin{array}{l}
\displaystyle \frac{\alpha}{16}\|\hat x^k-x^k\|^2
\leq  e_{\alpha^{-1}\psi^k}(x^k)-e_{\alpha^{-1}\psi^{k+1}}(x^{k+1})+
\gamma_4 \sigma,
\end{array}
$$
or equivalently (noting $\|\hat x^k-x^k\|^2=\alpha^{-2}\|\nabla e_{\alpha^{-1}\psi^k}(x^k)\|^2$)
$$
\begin{array}{l}
 \|\nabla e_{\alpha^{-1}\psi^k}(x^k)\|^2
\leq 16 \alpha \left[ e_{\alpha^{-1}\psi^k}(x^k)-e_{\alpha^{-1}\psi^{k+1}}(x^{k+1})+
\gamma_4 \sigma \right],
\end{array}
$$
namely (\ref{eqfinalesti}) holds.
\hfill $\Box$

The following result is crucial for estimating the constraint violation.
\begin{proposition}\label{prop:cregret}
Let $(x^t,\lambda^t)$ be generated by QPALM and Assumptions (A3) and  (B1) be satisfied.
Then for  $i=1,\ldots,p$,
\begin{equation}\label{eq:ccomineq0}
\begin{array}{ll}
\displaystyle \sum_{t=1}^T g_i(x^t)
&\leq
\displaystyle \frac{1}{\sigma} \lambda^{T+1}_i+\left[\kappa_g+\displaystyle \frac{\kappa_{H}D_0}{2}\right]\sum_{t=1}^T\|x^{t+1}-x^t\|.
\end{array}
\end{equation}
\end{proposition}
{\bf Proof}.
From the definition $\lambda^{t+1}_i=[\lambda^t_i+\sigma q^t_i(x^{t+1})]_+$, we have  that
$$
\begin{array}{ll}
\lambda^{t+1}_i
&\geq \lambda^t_i+\sigma\left( g_i(x^t)+\langle \nabla_x g_i(x^t), x^{t+1}-x^t\rangle +\displaystyle \frac{1}{2} \left \langle\Sigma^t_i (x^{t+1}-x^t),x^{t+1}-x^t \right \rangle \right)\\[10pt]
& \geq \lambda^t_i+\sigma\left(g_i(x^t)- \|\nabla_x g_i(x^t)\|\|x^{t+1}-x^t\|-\displaystyle \frac{1}{2} \|\Sigma^t_i\|\| x^{t+1}-x^t\|^2\right),
\end{array}
$$
which, from Assumptions (B1) and (A3), implies that
$$
\begin{array}{ll}
\displaystyle \sum_{t=1}^T g_i(x^t)&\leq  \displaystyle \frac{1}{\sigma} \lambda^{T+1}_i+ \sum_{t=1}^T\|\nabla_x g_i(x^t)\|\|x^{t+1}-x^t\|+ \displaystyle \frac{1}{2} \sum_{t=1}^T\|\Sigma^t_i\|\|x^{t+1}-x^t\|^2\\[10pt]
&\leq
\displaystyle \frac{1}{\sigma} \lambda^{T+1}_i+ \kappa_g\sum_{t=1}^T\|x^{t+1}-x^t\|+\displaystyle \frac{\kappa_{H}D_0}{2}\sum_{t=1}^T\|x^{t+1}-x^t\|,
\end{array}
$$
which is just (\ref{eq:ccomineq0}). \hfill $\Box$\\
The following proposition is about the total complementarity violation.
\begin{proposition}\label{prop:compregret}
Let $(x^t,\lambda^t)$ be generated by QPALM and Assumptions (A2), (A3) and (B1), (B3) and (B4)be satisfied.
Then
\begin{equation}\label{eq:comineq0}
\begin{array}{ll}
-\displaystyle \sum_{t=1}^T\langle \lambda^t, g(x^t)\rangle
&\leq
\displaystyle \frac{1}{2\sigma} [\|\lambda^1\|^2-\|\lambda^{T+1}\|^2]+\displaystyle \frac{\sigma}{2}\sum_{t=1}^T\|g(x^t)\|^2+
\displaystyle \frac{1}{2\alpha}\sum_{t=1}^T\|\nabla_x f(x^t)\|^2
\end{array}
\end{equation}
and
\begin{equation}\label{eq:comcineqb}
-\left( \displaystyle \sum_{t=1}^T\langle \lambda^t, g(x^t) \rangle\right)\leq \displaystyle \frac{\sigma}{2}\nu_g^2T+\displaystyle \frac{1}{2\alpha}\kappa_f^2T.
\end{equation}
\end{proposition}
{\bf Proof}.
It follows from (\ref{eq:opt-x-2}) that
$$
\begin{array}{ll}
\displaystyle\langle \nabla_x f(x^t),x^{t+1}-x^t \rangle + \frac{1}{2\sigma}\|\lambda^{t+1}\|^2
+ \alpha \|x^{t+1}-x^t\|^2+\displaystyle \frac{1}{2}\|x^{t+1}-x^t\|^2_{\Sigma^t_0}\\[10pt]
\leq  \displaystyle\frac{1}{2\sigma}\|[\lambda^t+\sigma g(x^t)]_+\|^2\leq \displaystyle\frac{1}{2\sigma}\|[\lambda^t+\sigma g(x^t)]\|^2
\end{array}
$$
which implies
$$
\begin{array}{ll}
-\langle \lambda^t, g(x^t)\rangle & \leq  \displaystyle \frac{1}{2\sigma}[\|\lambda^t\|^2-\|\lambda^{t+1}\|^2] -\displaystyle\langle \nabla_x f(x^t),x^{t+1}-x^t \rangle\\[10pt]
&\quad - \alpha \|x^{t+1}-x^t\|^2-\displaystyle \frac{1}{2}\|x^{t+1}-x^t\|^2_{\Sigma^t_0}
+ \displaystyle\frac{\sigma}{2}\|
g(x^t)\|^2\\[10pt]
&\leq \displaystyle \frac{1}{2\sigma}[\|\lambda^t\|^2-\|\lambda^{t+1}\|^2]+\displaystyle \frac{1}{2\alpha}\|\nabla_x f(x^t)\|^2\\[10pt]
& \quad \quad -\displaystyle \frac{\alpha}{2} \|x^{t+1}-x^t\|^2-\displaystyle \frac{1}{2}\|x^{t+1}-x^t\|^2_{\Sigma^t_0}
+ \displaystyle\frac{\sigma}{2}\| g(x^t)\|^2.
\end{array}
$$
Making a sum from $1$ to $T$, we obtain (\ref{eq:comineq0}).
One gets  (\ref{eq:comcineqb})  from (\ref{eq:comineq0}).\hfill $\Box$

\begin{lemma}\label{lemTwh}
If  $T$ be a positive integer satisfying
 \begin{equation}\label{eqTg}
 T > \max \left\{\left(L_0+\displaystyle\frac{1}{2}\kappa_{\Sigma}\right)^3,\displaystyle \frac{1}{2}\left[2\nu_gp^{1/2}\kappa_{H}+2p\kappa_g^2+\displaystyle \frac{1}{2}p\kappa_{H}^2D_0^2  \right],\displaystyle \frac{p(\kappa_g+\kappa_{H}D_0/2)^2}{16\gamma_5}
   \right\},
 \end{equation}
 then
 \begin{equation}\label{hatbN}
 \widehat \beta_T(T^{-2/3})\leq \gamma_5 T^{1/3},
 \end{equation}
 where
\begin{equation}\label{eq:alp}
\gamma_5=\left[2+p^{1/2}\kappa_{H}\gamma_1+\displaystyle \sum_{j=1}^pL_j\gamma_2\right].
\end{equation}
Moreover, if $\alpha =16\gamma_5 T^{1/3}$,$\sigma=T^{-2/3}$, then
\begin{equation}\label{eqxd}
\|x^{t+1}-x^t\| \leq \gamma_6 T^{-1/3},
\end{equation}
where
\begin{equation}\label{eq:g6}
 \gamma_6=\displaystyle\frac{1}{8\gamma_5}
\left(\kappa_f+(\kappa_g+\kappa_{H}D_0/2)\sqrt{p}
[\gamma_3(16\gamma_5, M^*)+\nu_g]\right).
\end{equation}
\end{lemma}
 {\bf Proof}.
It is easy to get from the definition of $\widehat \beta_k(\sigma)$ that
\begin{equation}\label{betahatdef0}
\begin{array}{ll}
\widehat \beta_T(T^{-2/3})& =\beta_T(T^{-2/3})+\displaystyle\frac{1}{2}\kappa_{\Sigma}+\displaystyle \frac{T^{-2/3}}{2}\left[2\nu_gp^{1/2}\kappa_{H}+2p\kappa_g^2+\displaystyle \frac{1}{2}p\kappa_{H}^2D_0^2  \right]+p^{1/2}\kappa_{H}\gamma_1T^{1/3}\\[6pt]
&=\left(L_0+\displaystyle\frac{1}{2}\kappa_{\Sigma}\right)+\left(p^{1/2}\kappa_{H}\gamma_1+\displaystyle \sum_{j=1}^pL_j\gamma_2\right)T^{1/3}\\[10pt]
&\quad \quad +\displaystyle \frac{T^{-2/3}}{2}\left[2\nu_gp^{1/2}\kappa_{H}+2p\kappa_g^2+\displaystyle \frac{1}{2}p\kappa_{H}^2D_0^2  \right].
\end{array}
\end{equation}
and obtain (\ref{hatbN}).  It is easy to verify  $\alpha \geq p(\kappa_g+\kappa_{H}D_0/2)^2\sigma>0$ and we obtain from (\ref{eq:diffXa}) that
\[
\begin{array}{ll}
\|x^{t+1}-x^t\|&\leq \displaystyle\frac{2}{\alpha}
\left(\kappa_f+(\kappa_g+\kappa_{H}D_0/2)\sqrt{p}
[\|\lambda^t\|+\nu_g\sigma]\right)\leq \gamma_6 T^{-1/3}.
\end{array}
\]
The proof is completed.
\hfill $\Box$

\begin{theorem}\label{th:mainth}
Let (A1)--- (A5) hold and
\begin{equation}\label{conds}
\begin{array}{l}
 \mbox{$\Sigma^k_i \prec -L_i I$ for $i=1,\ldots p$},\\[6pt]
 \mbox{$\Sigma^k_0 = \sum_{j=1}^p \lambda_j^k \Sigma_j^k + I \succ 0$}.
\end{array}
 \end{equation}
 Let $T$ be a positive integer satisfying
 \begin{equation}\label{eqTg}
 T > \max \left\{\left(L_0+\displaystyle\frac{1}{2}\kappa_{\Sigma}\right)^3,\displaystyle \frac{1}{2}\left[2\nu_gp^{1/2}\kappa_{H}+2p\kappa_g^2+\displaystyle \frac{1}{2}p\kappa_{H}^2D_0^2 \right], \displaystyle \frac{p(\kappa_g+\kappa_{H}D_0/2)^2}{16\gamma_5}
   \right\}.
\end{equation}
 If $\alpha =16\gamma_5T^{1/3}$ and $\sigma=T^{-2/3}$, then
 \begin{itemize}
 \item[(a)]The iteration complexity of  Lagrangian  gradient measure is
 \begin{equation}\label{eq:KKTg}
 \displaystyle \frac{1}{T}
 \displaystyle \sum_{k=1}^T \|\nabla e_{\alpha^{-1}\psi^k}(x^k)\|^2 \leq 16^2\gamma_5\left\{[f(x^1)-f_{*}+\gamma_3(16\gamma_5, M^*)\nu_g]T^{-2/3}+\gamma_4T^{-1/3}\right\},
 \end{equation}
 where $f_{*}=\min_{x \in X} f(x)$.
 \item[(b)] The iteration complexity of constraint violation is
 \begin{equation}\label{eq:consv}
 \displaystyle \frac{1}{T} \sum_{t=1}^T g_i(x^t)\leq \left\{ \gamma_3(16\gamma_5, M^*)+\left(\kappa_g+\displaystyle
 \frac{\kappa_{H}D_0}{2}\right)\gamma_6\right\}T^{-1/3}
 \end{equation}
 \item[(c)] The iteration complexity of complementarity violation is
 \begin{equation}\label{eq:comreg}
-\displaystyle \frac{1}{T}\left( \displaystyle \sum_{t=1}^T\langle \lambda^t, g(x^t) \rangle\right)\leq \displaystyle \frac{\nu_g^2}{2}T^{-2/3}+\displaystyle \frac{1}{32\gamma_5}\kappa_f^2T^{-1/3}.
\end{equation}
 \end{itemize}
\end{theorem}
{\bf Proof}. Since Assumptions (A1)--(A5) hold, in view of (\ref{conds}), from Proposition \ref{corPosiSig}, we know from conditions (B1)--(B4) hold.
From the choice $\alpha=16\gamma_5T^{1/3}$ and the condition (\ref{eqTg}) of $T$, we know (\ref{eqal1}) holds, namely
$$
\alpha> 16 \hat \beta_k(\sigma).
$$
Then from Theorem \ref{th:L-est}, we obtain (\ref{eqfinalesti}), namely
$$
\begin{array}{l}
 \|\nabla e_{\alpha^{-1}\psi^k}(x^k)\|^2
\leq 16 \alpha \left[ e_{\alpha^{-1}\psi^k}(x^k)-e_{\alpha^{-1}\psi^{k+1}}(x^{k+1})+
\gamma_4 \sigma \right].
\end{array}
$$
Taking the sum over $1,2,\ldots, T$, divided by $T$, we have
$$
\begin{array}{l}
 \displaystyle \frac{1}{T}\sum_{k=1}^T\|\nabla e_{\alpha^{-1}\psi^k}(x^k)\|^2
\leq 16 \alpha \displaystyle \frac{1}{T} \left[ e_{\alpha^{-1}\psi^1}(x^1)-e_{\alpha^{-1}\psi^{T+1}}(x^{T+1})\right]+
16\gamma_4\alpha \sigma.
\end{array}
$$
Since $e_{\alpha^{-1}\psi^1}(x^1)\leq f(x_1)$ and for $\hat x^{T+1}=
{\rm prox}_{\alpha^{-1}\psi^{T+1}}(x^{T+1})$,
$$
\begin{array}{ll}
e_{\alpha^{-1}\psi^{T+1}}(x^{T+1})& =L(\hat x^{T+1},\lambda^{T+1})+\displaystyle \frac{\alpha}{2}
\|\hat x^{T+1}-x^{T+1}\|^2\\[6pt]
& \geq
L(\hat x^{T+1},\lambda^{T+1})\\[6pt]
&=f(\hat x^{T+1})+\langle \lambda^{T+1}, g(\hat x^{T+1})\rangle
\geq  f_*-\gamma_3\nu_g.
\end{array}
$$
we have
\begin{equation}\label{diffe}
\left[ e_{\alpha^{-1}\psi^1}(x^1)-e_{\alpha^{-1}\psi^{T+1}}(x^{T+1})\right]
\leq f(x^1)-f^*+\gamma_3\nu_g.
\end{equation}
Then we obtain (\ref{eq:KKTg}) from the choices of $\alpha$ and $\sigma$.
In view of Proposition \ref{prop:cregret}, we have that (\ref{eq:ccomineq0}) holds, namely

\begin{equation}\label{eq:ccomineq0}
\begin{array}{ll}
\displaystyle \sum_{t=1}^T g_i(x^t)
&\leq
\displaystyle \frac{1}{\sigma} \lambda^{T+1}_i+\left[\kappa_g+\displaystyle \frac{\kappa_{H}D_0}{2}\right]\sum_{t=1}^T\|x^{t+1}-x^t\|.
\end{array}
\end{equation}
 From (\ref{eqxd}) of Lemma \ref{lemTwh}, we have $\|x^{t+1}-x^t\| \leq \gamma_6 T^{-1/3}$ and $\|\lambda^k\| \leq \gamma_3(16\gamma_5, M^*)$  from  (\ref{eq:lamB}) of  Corollary \ref{cor:ebound}.
Thus, we obtain from (\ref{eq:ccomineq0}) that
 $$
 \begin{array}{ll}
\displaystyle \frac{1}{T}\displaystyle \sum_{t=1}^T g_i(x^t)
&\leq
\displaystyle \frac{1}{T^{1/3}}\left\{\gamma_3(16\gamma_5, M^*)+\left[\kappa_g+\displaystyle \frac{\kappa_{H}D_0}{2}\right]\gamma_6\right\}.
\end{array}
 $$
 Namely, (\ref{eq:consv}) holds.
Estimate (\ref{eq:comreg}) comes from (\ref{eq:comcineqb}) of Proposition \ref{prop:compregret} and the choices of $\alpha$ and $\sigma$.
\hfill $\Box$

\section{Numerical Results}
\setcounter{equation}{0}
In this section, we report numerical results for the proposed proximal method of multipliers with quadratic approximations (QPALM) on three preliminary test problems. We compare QPALM with the proximal augmented Lagrangian method (pALM) in \cite{Adeoye2025} and the classical augmented Lagrangian method \cite{Hestenes1969}, since both methods are closely related to the augmented Lagrangian framework considered in this paper and thus provide natural baselines for comparison. All numerical experiments are implemented in MATLAB 2025a on a MacBook Pro equipped with an Apple M5 chip, 24 GB unified memory, and 1 TB SSD storage.
\subsection{Solving subproblems}
This subsection focuses on solving the subproblem in QPALM, which can be expressed as
\begin{equation}\label{eq:general-subp}
\begin{array}{ll}
\displaystyle   \min_{x \in X}  \phi(x) := & \langle \nabla_{x} f(x^{t}),x-x^t \rangle +\displaystyle \frac{1}{2}\langle \Sigma_{0}^{t} (x-x^t),x-x^t \rangle\\[10pt]
& +
\displaystyle \frac{1}{2\sigma}\displaystyle \sum_{i=1}^p\left[\lambda_{i}+\sigma  \left(\langle \nabla_{x}g_{i}(x^{t}),x-x^t \rangle +\displaystyle \frac{1}{2}\langle \Sigma_{i}^{t} (x-x^t),x-x^t \rangle\right)\right]_+^2+\displaystyle \frac{\alpha}{2}\|x-x^t\|^2.
\end{array}
\end{equation}
Problem (\ref{eq:general-subp}) is a strongly convex optimization problem and we can apply
Nesterov's accelerated projected gradient method (APG) (see \cite{Beck2017}) to solve it.
\begin{algorithm}[H]\label{Algorithm_sub}
	\textbf{Step 0}: Input $x^{0} \in X_{0}$ and $\eta >1$. Set $y^{0} = x^{0}$, $L_{-1} = 1$ and $k : =0$.
\textbf{Step 1}: Set
	\begin{equation}\nonumber
		x^{k+1} = T_{L_{k}} (y^{k}),
	\end{equation}
	where $T_{L} (y) := \Pi_{X_{0}} \left[y - \frac{1}{L}\nabla \Phi (y)\right]$, the stepsize $L_{k} = L_{k-1}\eta^{i_{k}}$ and $i_{k}$ is the smallest nonnegative integer satisfies the following condition
	\begin{equation}\nonumber
		\begin{aligned}
			\phi \left(T_{L_{k-1}\eta ^{i_{k}}}(y^{k})\right) \leq \phi(y^{k}) &+ \langle \nabla \phi (y^{k}), T_{L_{k-1}\eta ^{i_{k}}}(y^{k}) -y^{k} \rangle \\
			&+ \frac{L_{k-1}\eta ^{i_{k}}}{2} \Vert T_{L_{k-1}\eta ^{i_{k}}}(y^{k}) -y^{k} \Vert^2.
\end{aligned}
	\end{equation}
	
	\textbf{Step 2}: Compute
	\begin{equation}\nonumber
		y^{k+1} = x^{k+1} +\frac{k}{k+3}\left(x^{k+1}-x^{k}\right).
	\end{equation}
	
	\textbf{Step 3}: Set $k :=k+1$ and go to Step 1.
	\caption{Nesterov's accelerated projected gradient method for subproblem}
\end{algorithm}

A well-known convergence result of the above method is that, if $\phi$ is $\mu$-strongly convex and $\nabla \phi$ is $L$-Lipschitz continuous, then $\phi(x^{k}) - \phi(x^*) \leq \mathcal{O}\left((1 - \sqrt{\mu /L})^{k}\right)$. See \cite{Beck2017} for a detailed discussion on this topic. Here, we assume that the set $X$ is simple such thay the projection $\Pi_{X_{0}}$ can be efficiently computed.
\subsection{Nonconvex Quadratically Constrained Quadratic Programming}
In this subsection, we consider a class of randomly generated nonconvex quadratically constrained quadratic programming (QCQP) problems of the form
\begin{equation}\label{eq:qcqp-test}
\begin{aligned}
\min_{x \in X}\quad & f(x):=\frac12 x^\top Q_0 x + c_0^\top x,\\
\text{s.t.}\quad & g_i(x):=\frac12 x^\top Q_i x + c_i^\top x - r_i \le 0,\quad i=1,\ldots,p,
\end{aligned}
\end{equation}
where $X=[-R,R]^n$. To generate the quadratic data, we construct each matrix $Q_i\in\mathbb{S}^n$, $i=0,1,\ldots,p$, in the form
\[
Q_i=U_i^\top D_iU_i,
\]
where $U_i\in\mathbb{R}^{n\times n}$ is a random orthogonal matrix obtained from the $QR$ factorization of a Gaussian random matrix, and $D_i$ is a diagonal matrix of prescribed eigenvalues. The objective matrix $Q_0$ is chosen to be indefinite with
\[
\lambda_{\min}(Q_0)\ge -L_{0},
\]
so that the objective function $f$ is $L_{0}$-weakly convex. More precisely, a prescribed proportion of the eigenvalues of $Q_0$ are sampled from $[-L_{0},-0.1]$, while the remaining eigenvalues are sampled from $[0.5,3]$. For the constraint matrices $Q_i$, $i=1,\ldots,p$, most of them are generated to be positive semidefinite by sampling all eigenvalues from $[0.5,3]$, while the others are generated to be indefinite with
\[
\lambda_{\min}(Q_i)\ge -L_{i}
\]
by sampling a small proportion of the eigenvalues from $[-L_{i},-0.1]$ and the remaining ones from $[0.5,3]$. The linear terms are generated as
\[
c_0=\tau_0\xi_0,\qquad c_i=\tau_g\xi_i,\quad i=1,\ldots,p,
\]
where $\xi_0,\xi_1,\ldots,\xi_p\sim\mathcal{N}(0,I_n)$ are independent Gaussian random vectors, and $\tau_0,\tau_g>0$ are prescribed scaling parameters. In order to guarantee the existence of a strictly feasible point, we first generate a random vector $x^\ast\in X$, and then define
\[
r_i=\frac12 (x^\ast)^\top Q_ix^\ast+c_i^\top x^\ast+\delta_i,\qquad i=1,\ldots,p,
\]
where $\delta_i>0$ is sampled from a prescribed positive interval. With this construction,
\[
g_i(x^\ast)=-\delta_i<0,\qquad i=1,\ldots,p,
\]
and hence $x^\ast$ is a strictly feasible point for all generated instances. In the experiments, the initial point is chosen as $x^0=x^\ast$ and the initial multiplier is set to $\lambda^0=0$. In the test, we set $n=80$, $p=30$ and $R=2$.

\begin{figure}[htbp]
    \centering
    \subfloat[Averaged Lagrangian gradient measure]{
        \includegraphics[width=0.48\textwidth]{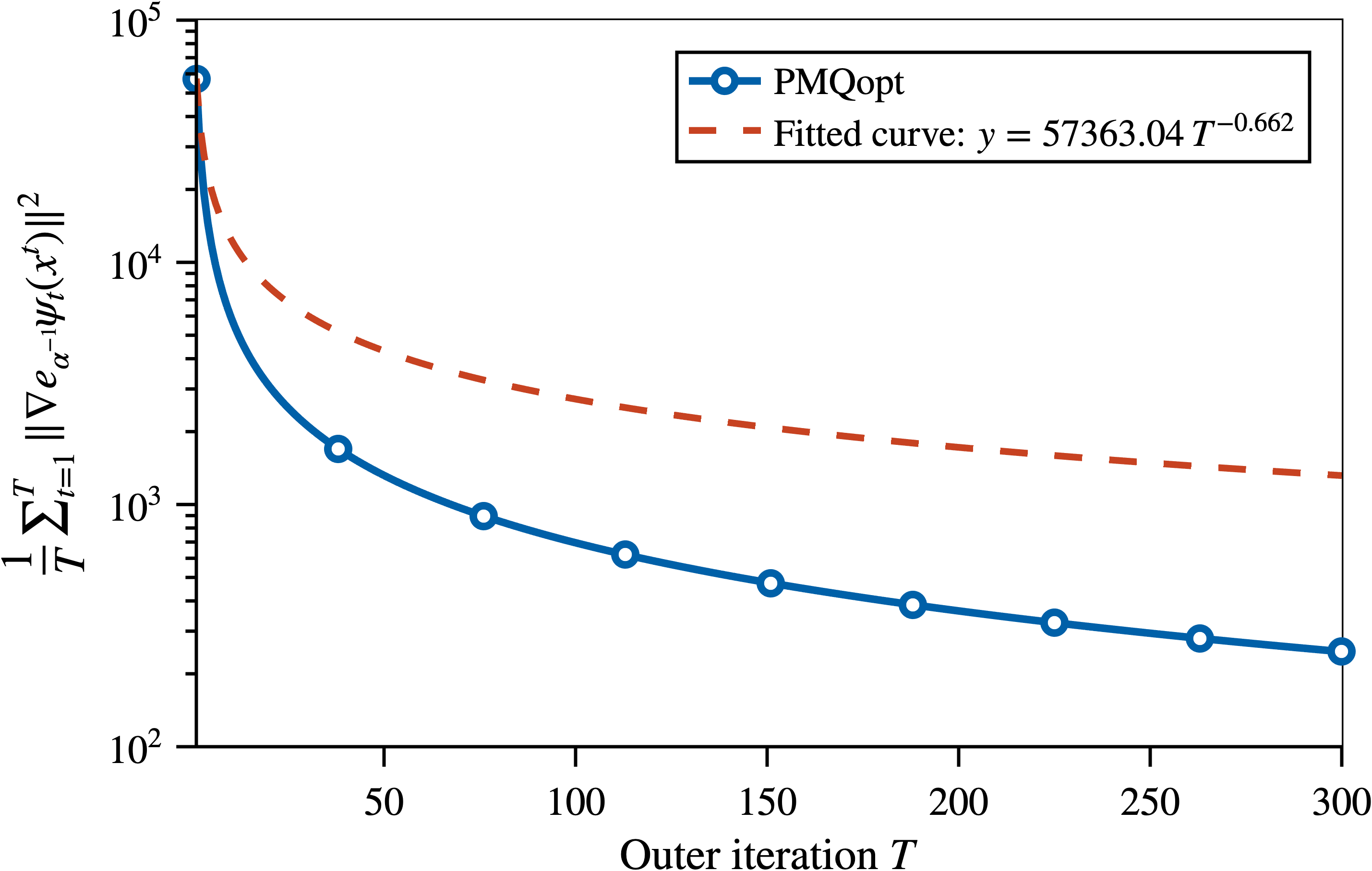}
        \label{fig:qcqp_grad}
    }
    \hfill
    \subfloat[Averaged constraint violation]{
        \includegraphics[width=0.48\textwidth]{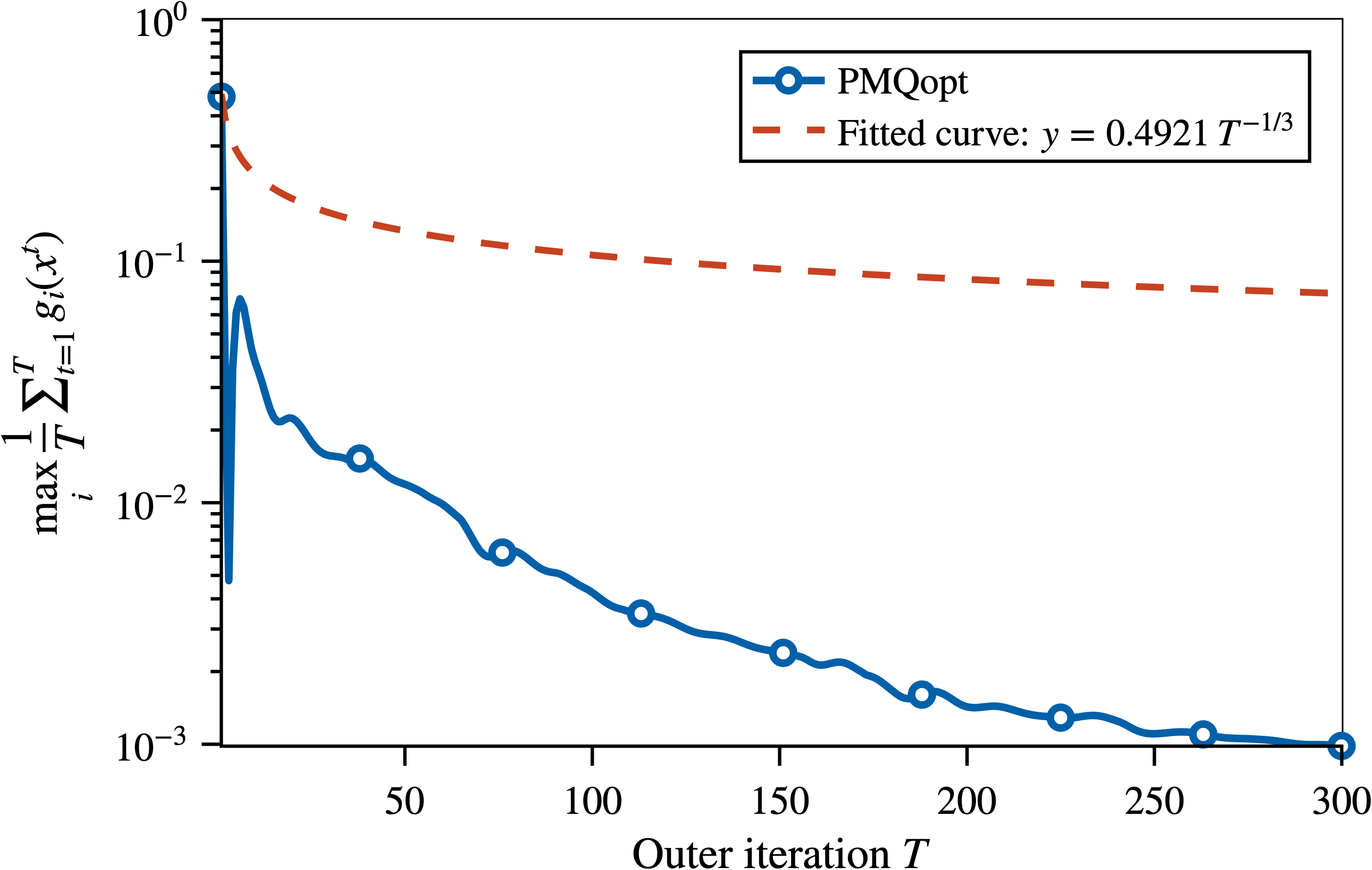}
        \label{fig:qcqp_feas}
    }

    \vspace{0.3cm}

    \subfloat[Averaged complementarity violation]{
        \includegraphics[width=0.48\textwidth]{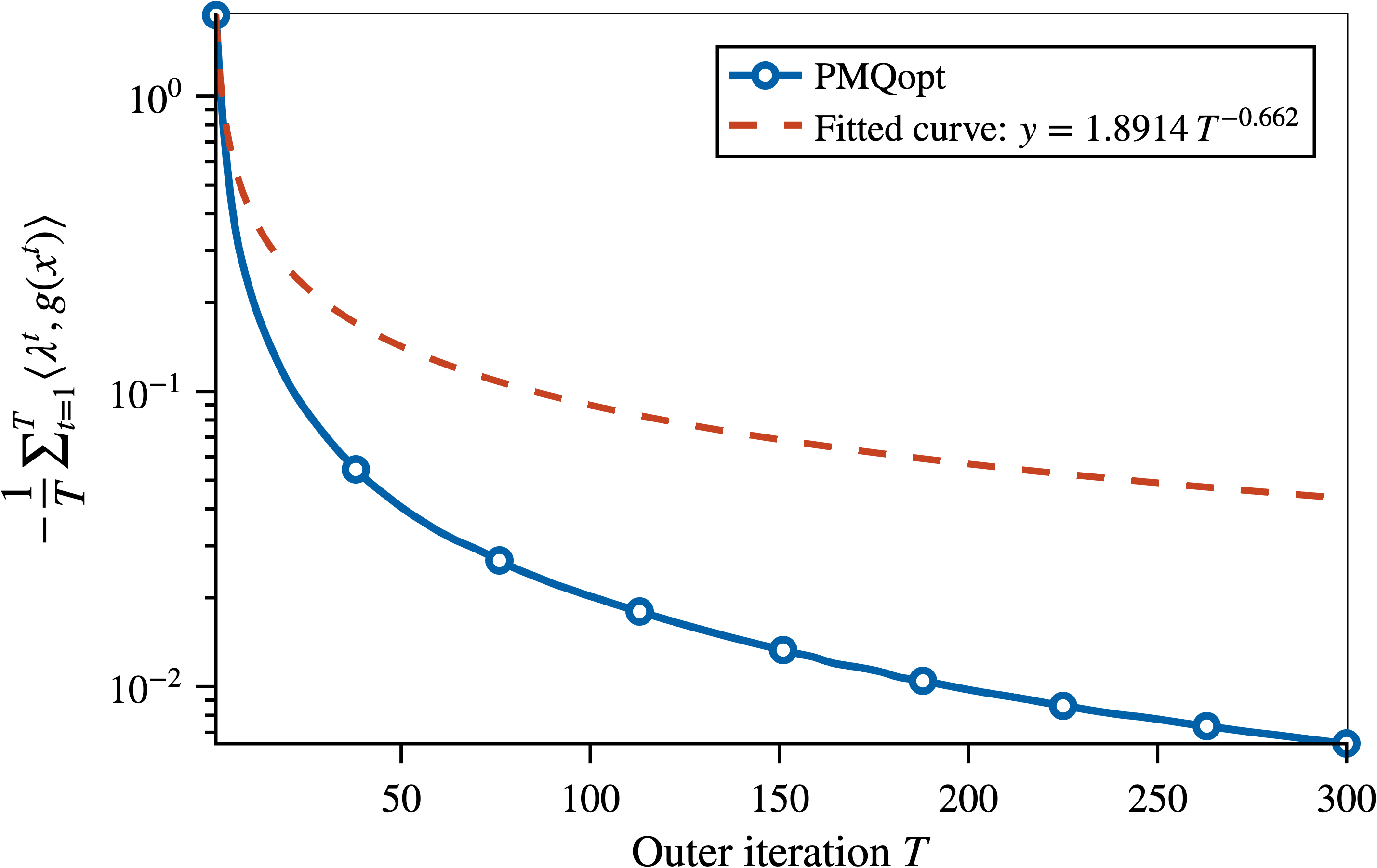}
        \label{fig:qcqp_comp}

}
    \hfill
    \subfloat[Data profile of CPU time]{
        \includegraphics[width=0.48\textwidth]{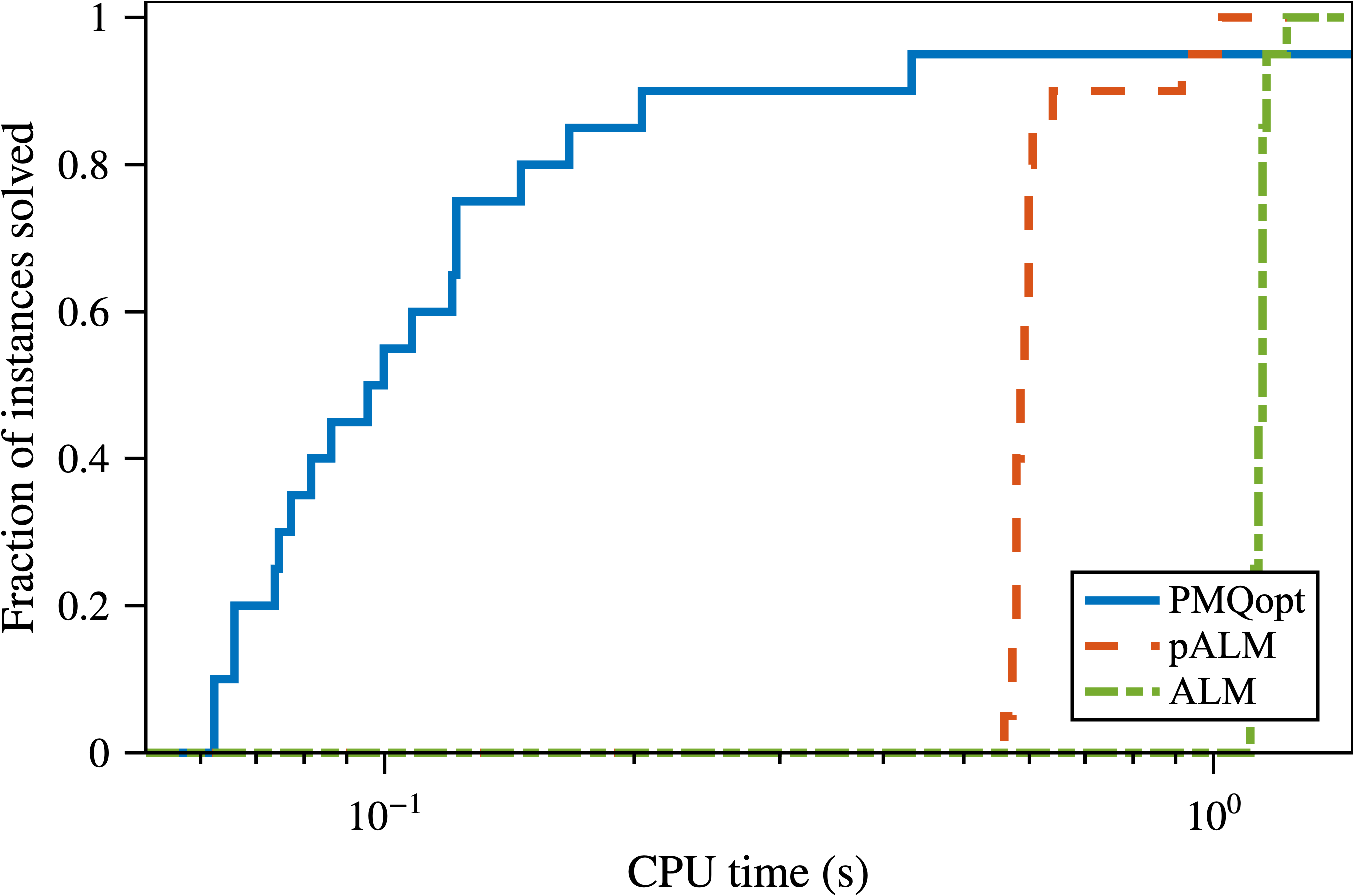}
        \label{fig:qcqp_profile}

    }
    \caption{Numerical performance of QPALM, pALM, and ALM on nonconvex QCQP instances. Panels (a)--(c) report the decay of three theoretical quantities for a representative instance, while panel (d) presents the data profile of CPU time over 20 randomly generated instances.}
    \label{fig:qcqp_main}
\end{figure}
Figures~\ref{fig:qcqp_grad}--\ref{fig:qcqp_comp} display the decay of the three theoretical quantities of QPALM on a representative nonconvex QCQP instance. The dashed lines are theory-consistent upper-envelope fits. All three quantities decrease steadily as the number of outer iterations increases. Moreover, the empirical decay rates of the first and third quantities are about $-0.662$, which falls into the theoretical interval $[-2/3,-1/3]$, while the second quantity is well captured by the rate $T^{-1/3}$. This is in agreement with Theorem~3.2.

Figure~\ref{fig:qcqp_profile} shows the data profile of CPU time over 20 randomly generated nonconvex QCQP instances. Here success means achieving at least $80\%$ of the best objective decrease attained by the three algorithms on a given instance. The QPALM curve rises earlier than those of pALM and ALM, indicating superior computational efficiency. In particular, QPALM is the fastest on 19 out of 20 instances, with median CPU time approximately $9.46\times 10^{-2}$ seconds, compared with $5.85\times 10^{-1}$ seconds for pALM and $1.13$ seconds for ALM.
\subsection{Nonconvex Neyman-Pearson classification}
In this subsection, we evaluate the proposed QPALM on a nonconvex Neyman--Pearson classification problem. The goal is to minimize the false-negative error subject to a constraint on the false-positive error. The model can be written as
\begin{equation}\label{eq:np-model}
\begin{aligned}
\min_{x \in \R^{d}} \quad & f(x) := \frac{1}{N_{0}} \sum_{i=1}^{N_{0}} \phi(x^{\top}a_{i}^{0}) \\
\mathrm{s.t.} \quad & g(x) := \frac{1}{N_{1}} \sum_{i=1}^{N_{1}} \phi(-x^{\top}a_{i}^{1}) - \tau \leq 0,
\end{aligned}
\end{equation}
where $\{a_i^0\}_{i=1}^{N_0}$ and $\{a_i^1\}_{i=1}^{N_1}$ denote the positive-class and negative-class samples, respectively, and $\tau$ specifies the admissible false-positive level. Here $\phi$ is the sigmoid function defined by
\[
\phi(u)=\frac{1}{1+\exp(u)}.
\]

The datasets used in our experiments are listed in Table~\ref{table 1}. For multi-class datasets, we manually convert them into binary classification tasks. For instance, for the MNIST dataset, we classify odd digits versus even digits.
\begin{table}[htbp]
\centering
\caption{Datasets used in Neyman--Pearson classification}
\label{table 1}
\begin{tabular}{cccccc}
\toprule[2pt]
Dataset & Data $N$ & Variable $n$ & Density & False-positive level $\tau$ & Reference\\
\midrule[1pt]
MNIST   & 60000 & 784  & 19.12\% & 0.2 & LeCun et al.~\cite{LeCun2010}\\
CINA    & 16033 & 132  & 29.56\% & 0.3 & Workbench Team~\cite{CINA2008}\\
gisette & 6000  & 5000 & 12.97\% & 0.2 & Guyon et al.~\cite{Guyon2004}\\
\bottomrule[2pt]
\end{tabular}
\end{table}

We test \eqref{eq:np-model} on the three benchmark datasets \texttt{CINA}, \texttt{MNIST}, and \texttt{gisette}. After separating the samples according to their labels, the corresponding triplets $(N_0,N_1,d)$ are $(3939,12094,132)$, $(23569,30508,784)$, and $(3000,3000,5000)$, respectively. The tolerance parameter is set to $\tau=0.3$ for \texttt{CINA} and $\tau=0.2$ for both \texttt{MNIST} and \texttt{gisette}. For QPALM, each inner subproblem is solved by the APG routine described in this paper, and the inner loop is terminated once
\[
\|x^{j+1}-x^j\|<10^{-6}.
\]
For theory validation, QPALM is run for $1000$ outer iterations. For algorithmic comparison, QPALM, pALM, and ALM are tested under the same outer stopping rule based on a KKT residual criterion, with a maximum budget of $120$ outer iterations. For each dataset, we present four figures. The first three are used to validate the theoretical quantities appearing in Theorem~3.2, while the fourth is used for algorithmic comparison. More precisely, for QPALM we consider the averaged stationarity measure, the averaged constraint violation, and the averaged complementarity measure, whose trajectories are shown in Figures~\ref{cina}--\ref{minist}. To compare the observed decay with Theorem~3.2, we overlay theorem-consistent upper-envelope fits. For the stationarity and complementarity measures, we use a model of the form
\[
y(T)=aT^\alpha,
\]
where $\alpha$ is restricted to the theoretical interval $[-2/3,-1/3]$, and $a$ is chosen as the smallest constant such that the fitted curve majorizes the entire empirical trajectory. For the constraint violation, we use the theorem-guided form
\[
y(T)=cT^{-1/3},
\]
where $c>0$ is chosen analogously. Hence, the dashed lines in the figures should be interpreted as theorem-consistent empirical upper bounds rather than unconstrained regressions. The numerical results are fully consistent with the complexity estimates in Theorem~3.2. For \texttt{CINA}, the fitted exponents for the stationarity and complementarity measures are approximately $-0.620$ and $-0.456$, respectively. For \texttt{MNIST}, the corresponding exponents are approximately $-0.601$ and $-0.662$, while for \texttt{gisette} they are approximately $-0.662$ and $-0.338$. All these values lie in the interval $[-2/3,-1/3]$, whereas the constraint-violation curves are well captured by the rate $T^{-1/3}$. This provides clear numerical support for the theoretical iteration-complexity bounds. The fourth figure reports the time-to-target comparison of QPALM, pALM, and ALM on the same dataset. Since each benchmark dataset corresponds to a single deterministic optimization problem, data profiles or performance profiles are not the most suitable choice. Instead, we use a time-to-target plot. Let
\[
\Delta_{\mathrm{alg}}(k):=
\frac{f(x^0)-\min_{0\le t\le k} f(x^t)}
     {f(x^0)-f_{\mathrm{ref}}},
\]
where $f_{\mathrm{ref}}$ denotes the best objective value attained by all three algorithms on the same dataset. For each target level $\theta\in[0.1,0.8]$, we record the smallest CPU time required by each algorithm to satisfy
\[
\Delta_{\mathrm{alg}}(k)\ge \theta.
\]

For each dataset, we present four figures. The first three are devoted to theory validation through the quantities appearing in Theorem~3.2, whereas the fourth is devoted to algorithmic comparison in terms of time-to-target performance. The resulting curves are shown in Figure~\ref{cina}, Figure~\ref{gisette} and Figure~\ref{minist}. A lower curve indicates better computational efficiency. On all three datasets, the curve of QPALM lies below those of pALM and ALM over a wide range of target levels, showing that QPALM reaches the same relative objective reduction with substantially less CPU time. This demonstrates the practical effectiveness of QPALM and highlights its computational advantage on the nonconvex Neyman--Pearson classification problems. \begin{figure}[htbp]
    \centering
    \subfloat[Averaged Lagrangian gradient measure]{
        \includegraphics[width=0.48\textwidth]{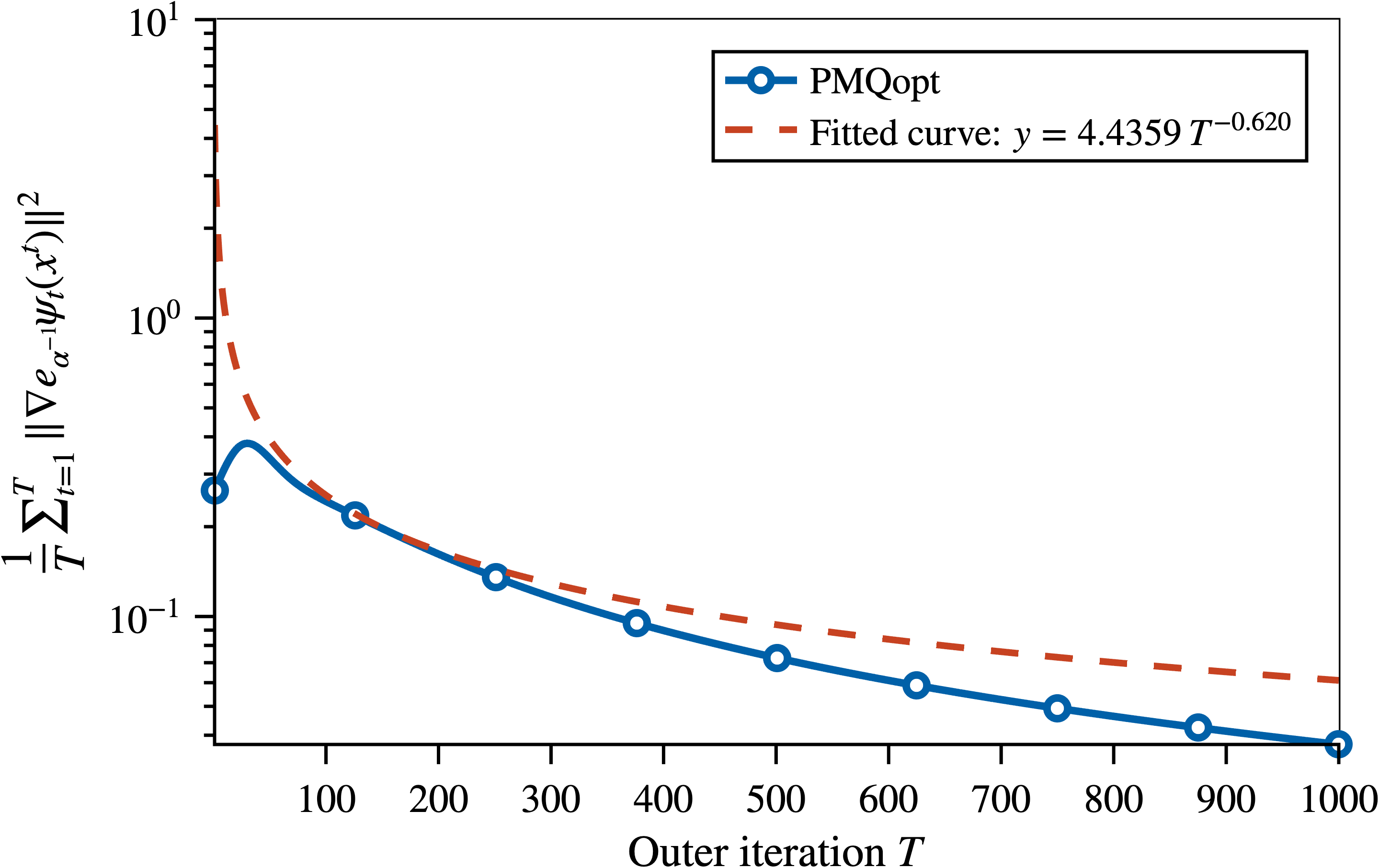}
        \label{fig:cina1}
    }
    \hfill
    \subfloat[Averaged constraint violation]{
        \includegraphics[width=0.48\textwidth]{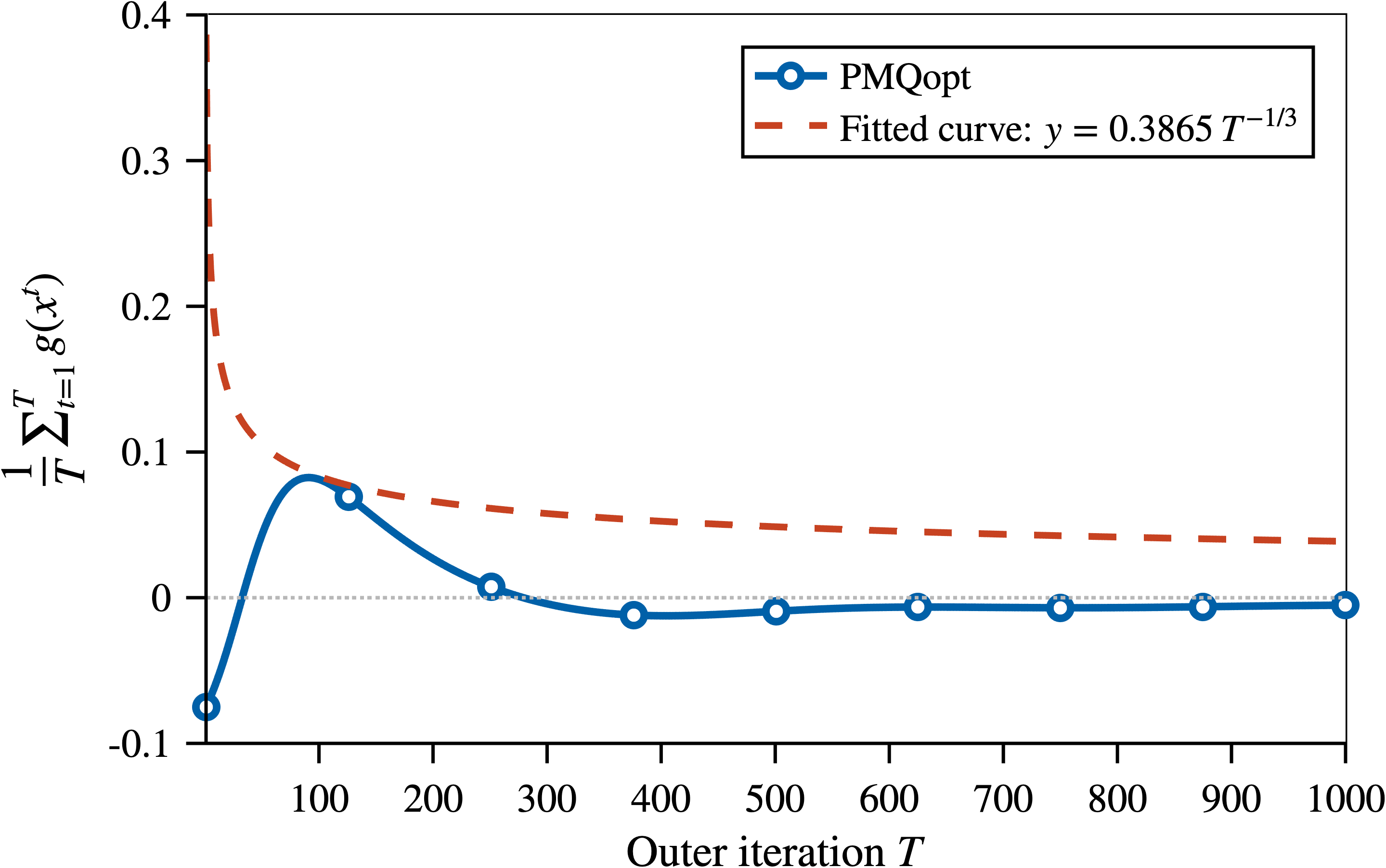}
        \label{fig:cina2}
    }

    \vspace{0.3cm}

    \subfloat[Averaged complementarity violation]{
        \includegraphics[width=0.48\textwidth]{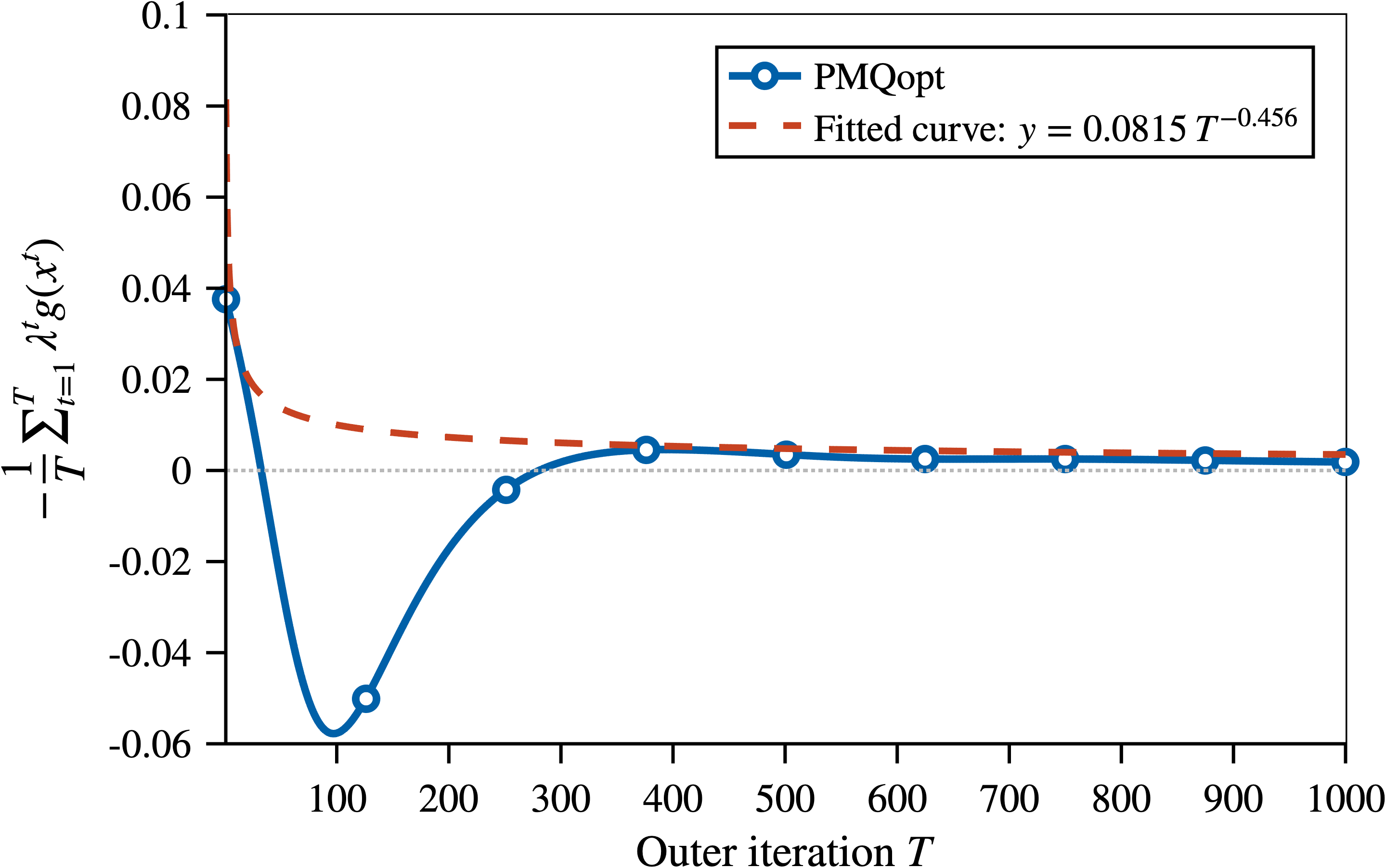}
        \label{fig:cina3}

    }
    \hfill
    \subfloat[Time-to-target
comparison]{
        \includegraphics[width=0.48\textwidth]{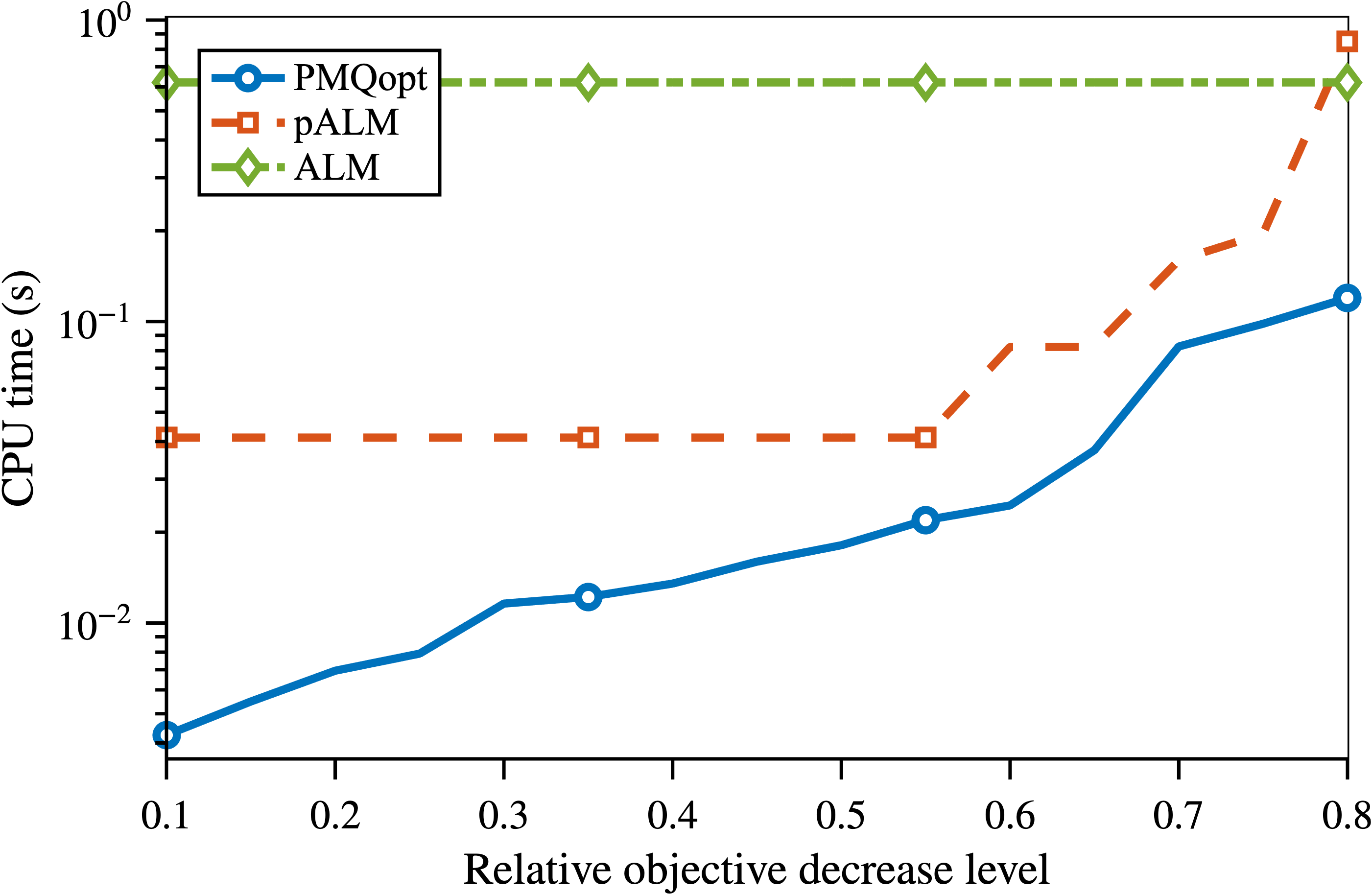}
        \label{fig:cina}

    }
    \caption{Theory validation and algorithmic comparison on the CINA dataset.}
    \label{cina}
\end{figure}

\begin{figure}[htbp]
    \centering
    \subfloat[Averaged Lagrangian gradient measure]{
        \includegraphics[width=0.48\textwidth]{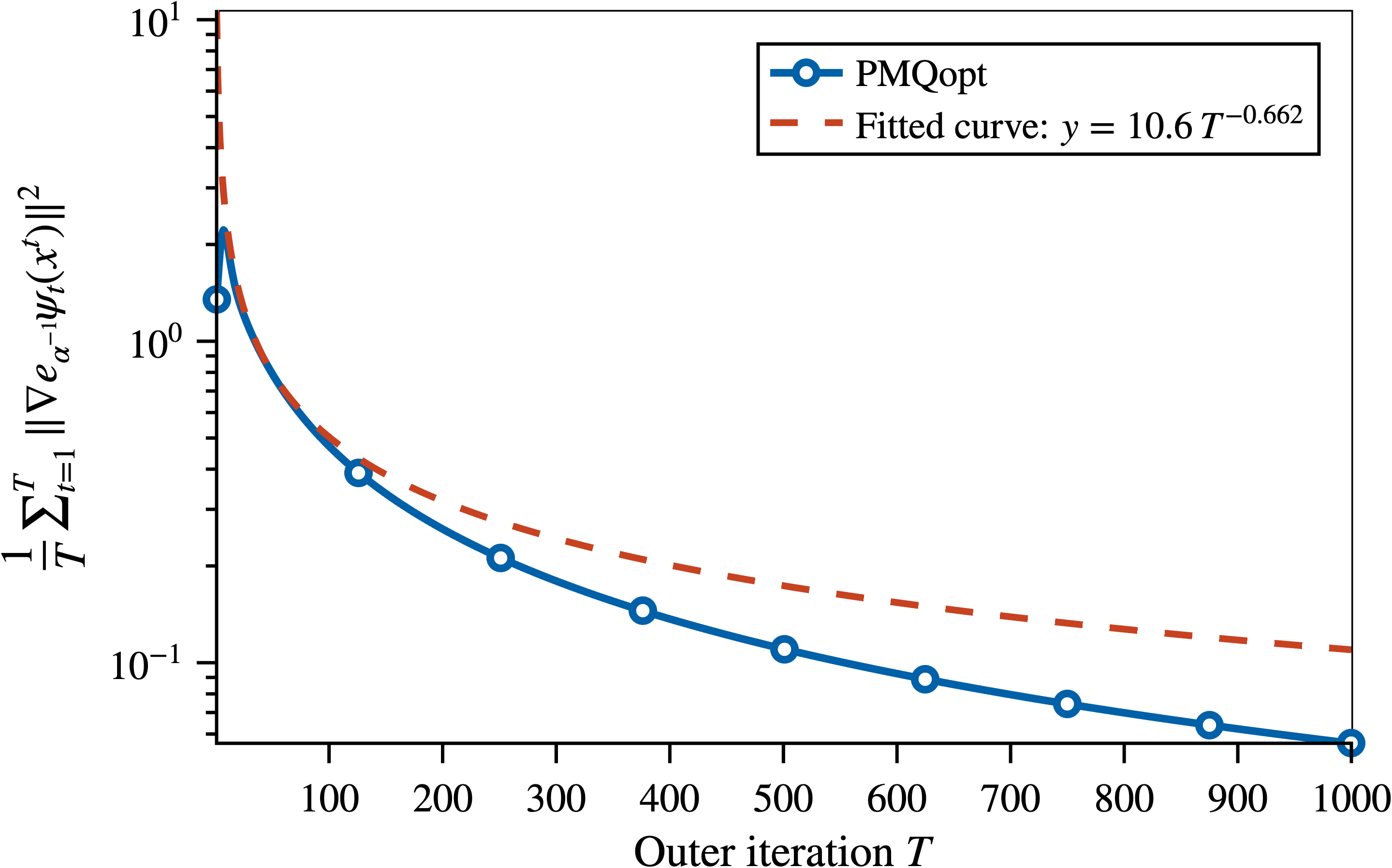}
        \label{fig:gisette1}
    }
    \hfill
    \subfloat[Averaged constraint violation]{
        \includegraphics[width=0.48\textwidth]{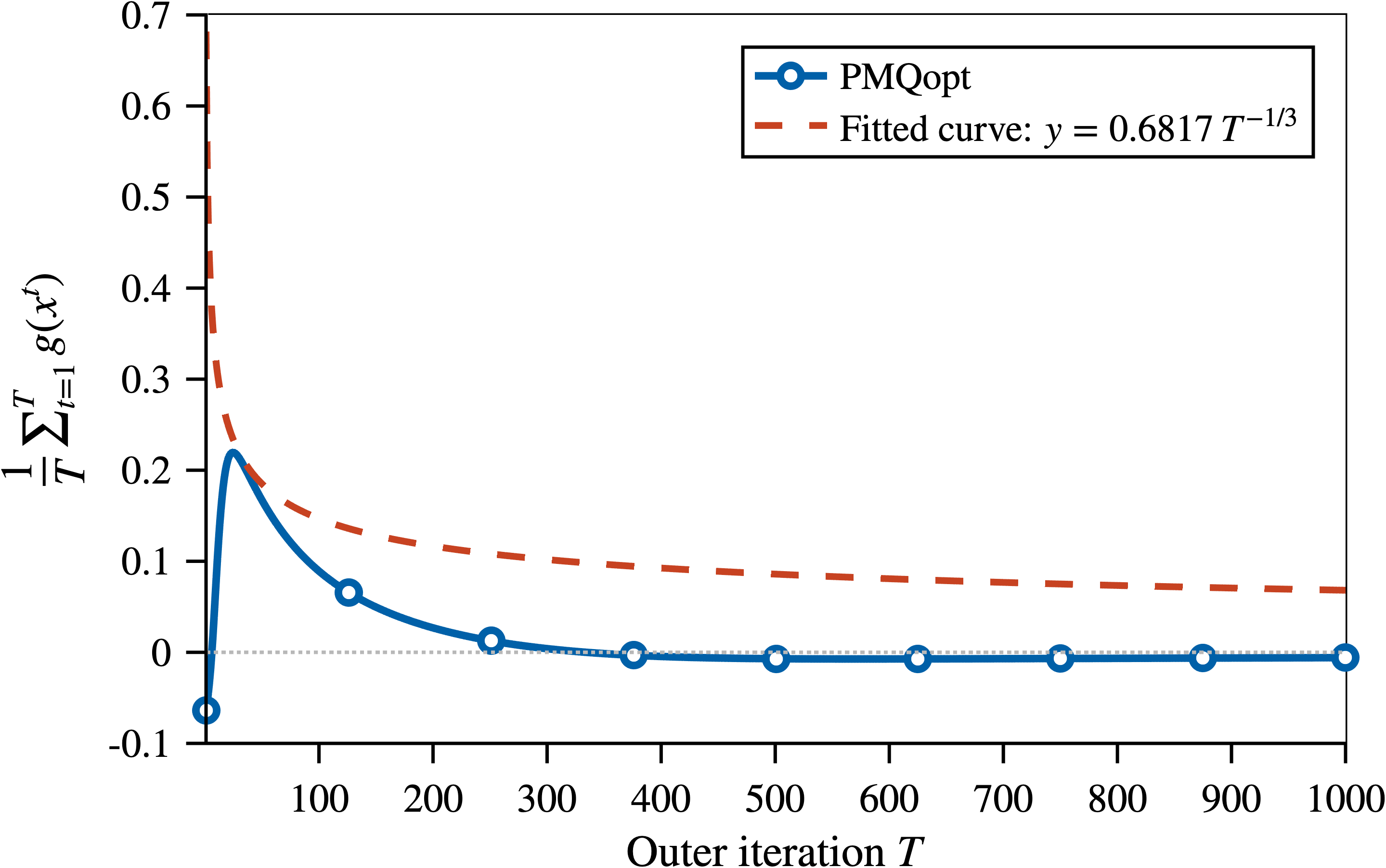}
        \label{fig:gisette2}
    }

 \vspace{0.3cm}

    \subfloat[Averaged complementarity violation]{
        \includegraphics[width=0.48\textwidth]{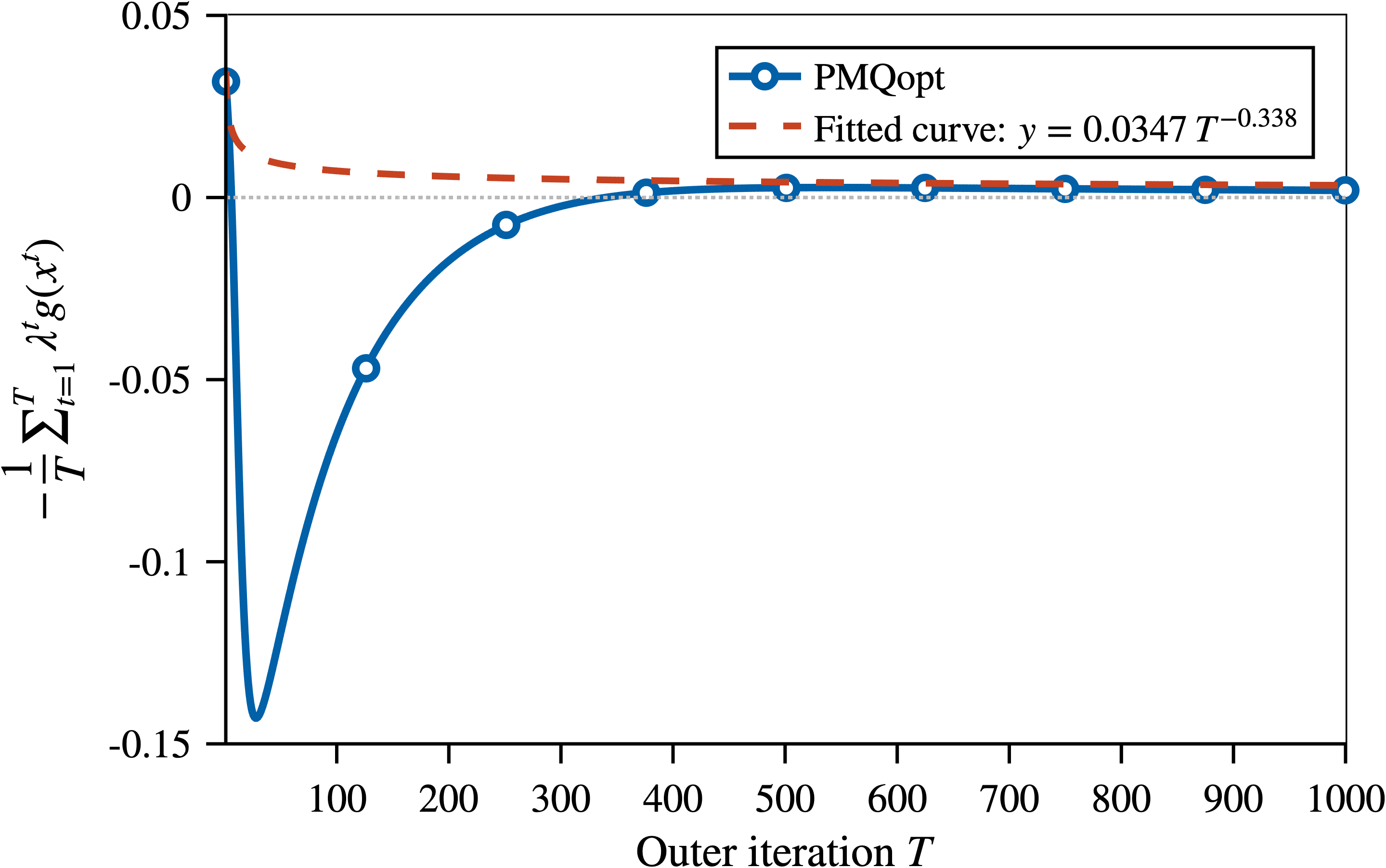}
        \label{fig:gisette3}

    }
    \hfill
    \subfloat[Time-to-target comparison]{
        \includegraphics[width=0.48\textwidth]{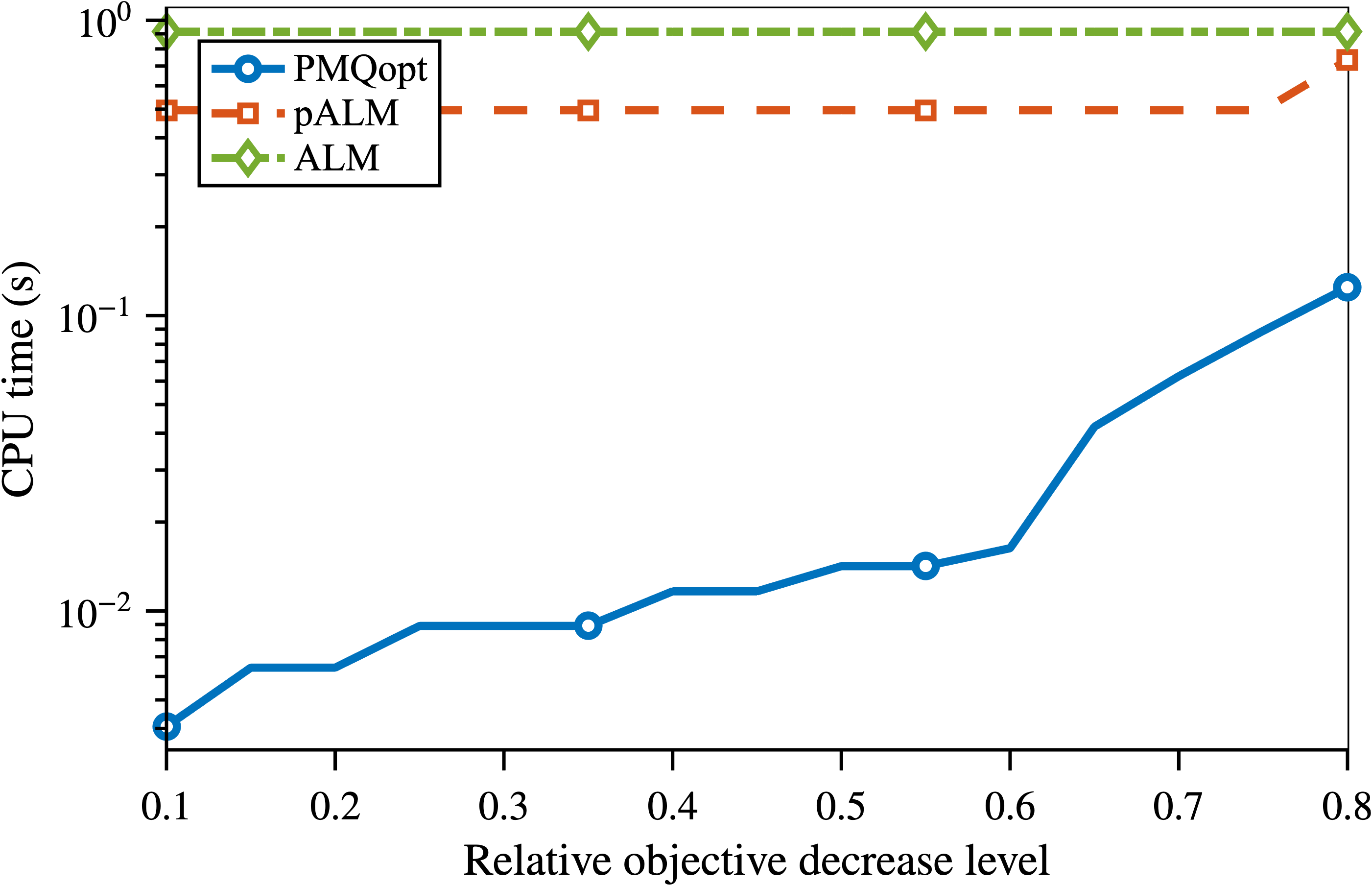}
        \label{fig:gisette}

    }
    \caption{Theory validation and algorithmic comparison on the gisette dataset.}
    \label{gisette}
\end{figure}

\begin{figure}[htbp]
    \centering
    \subfloat[Averaged Lagrangian gradient measure]{
        \includegraphics[width=0.48\textwidth]{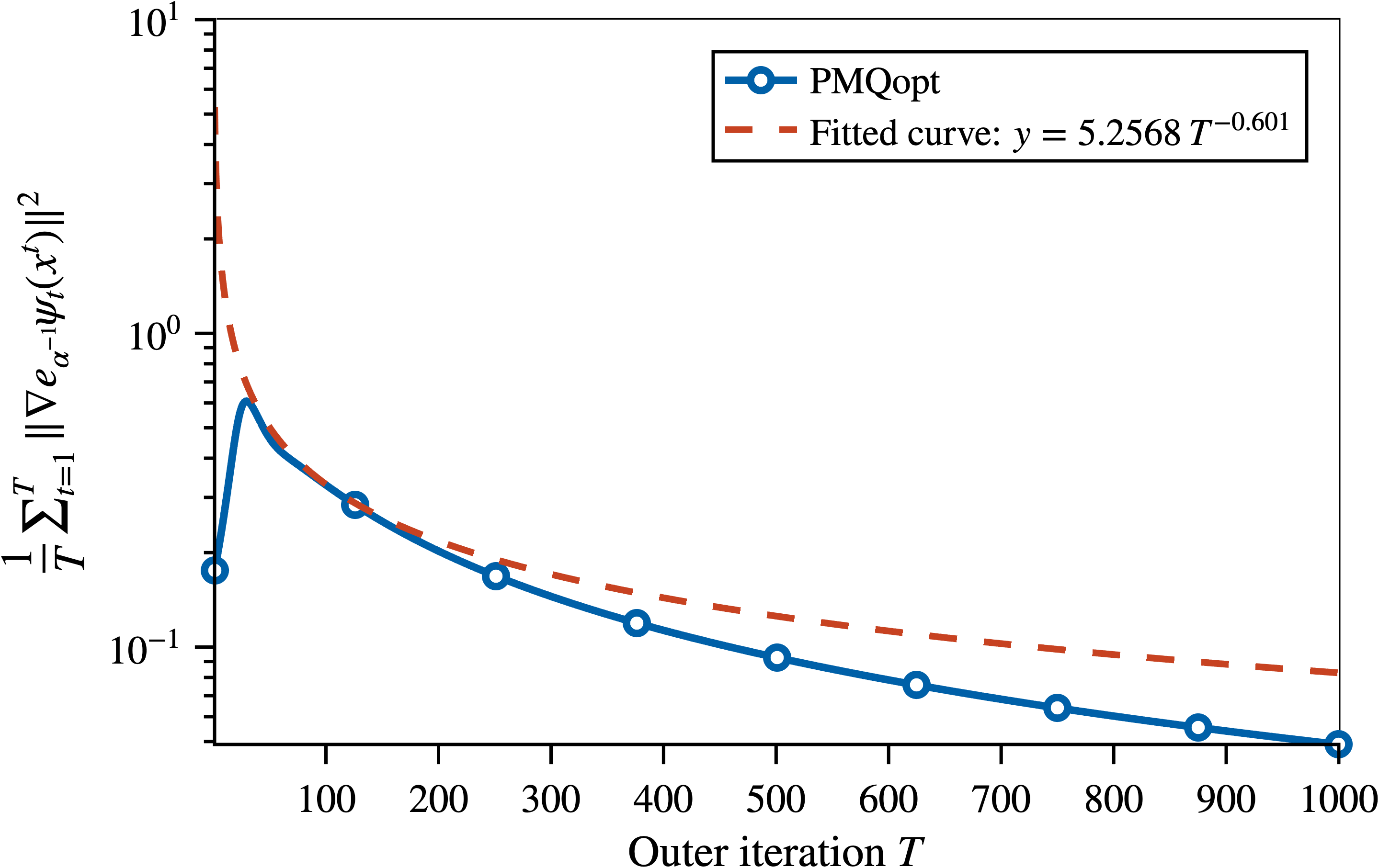}
        \label{fig:minist1}

   }
    \hfill
    \subfloat[Averaged constraint violation]{
        \includegraphics[width=0.48\textwidth]{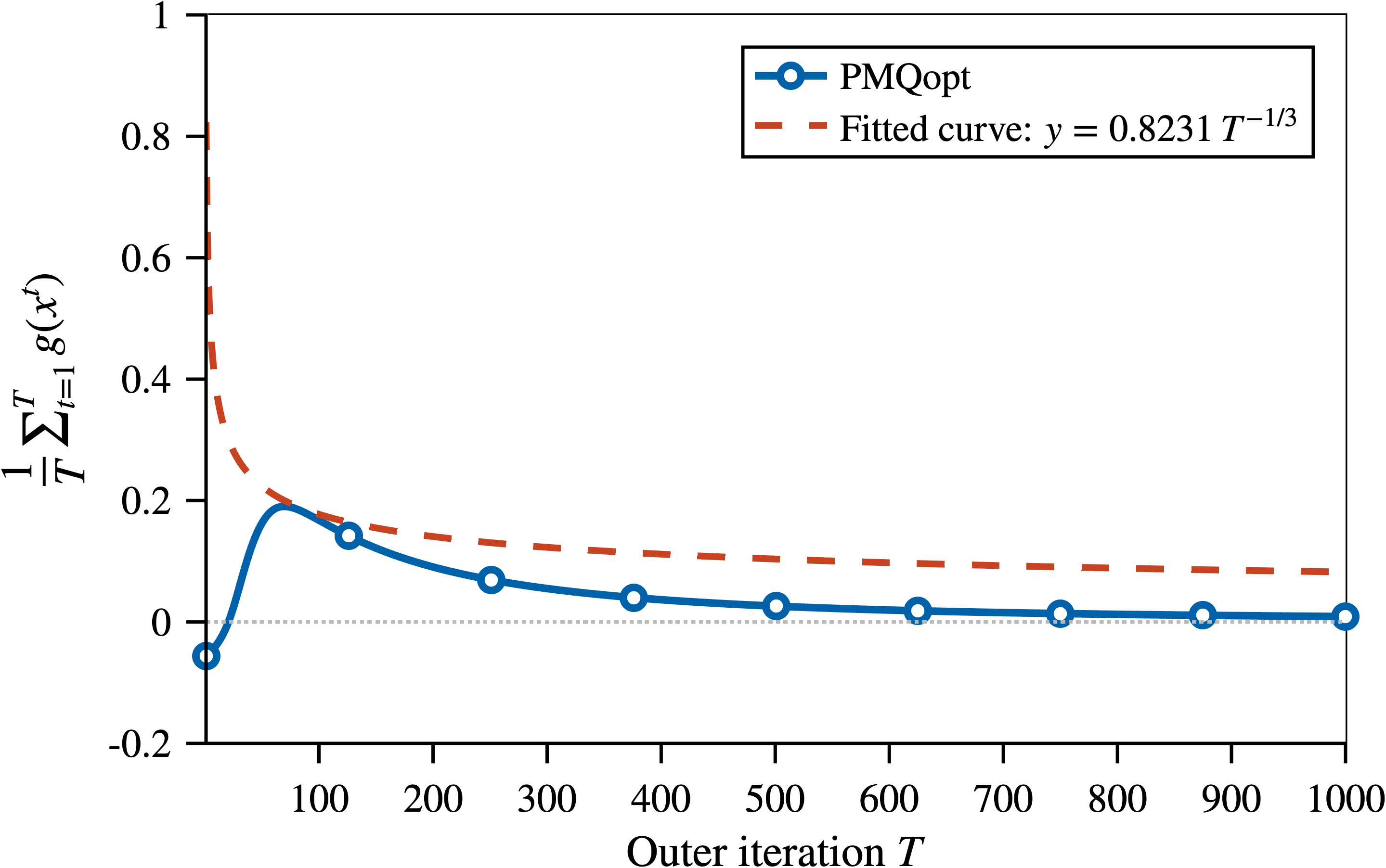}
        \label{fig:minist2}
    }

    \vspace{0.3cm}

    \subfloat[Averaged complementarity violation]{
        \includegraphics[width=0.48\textwidth]{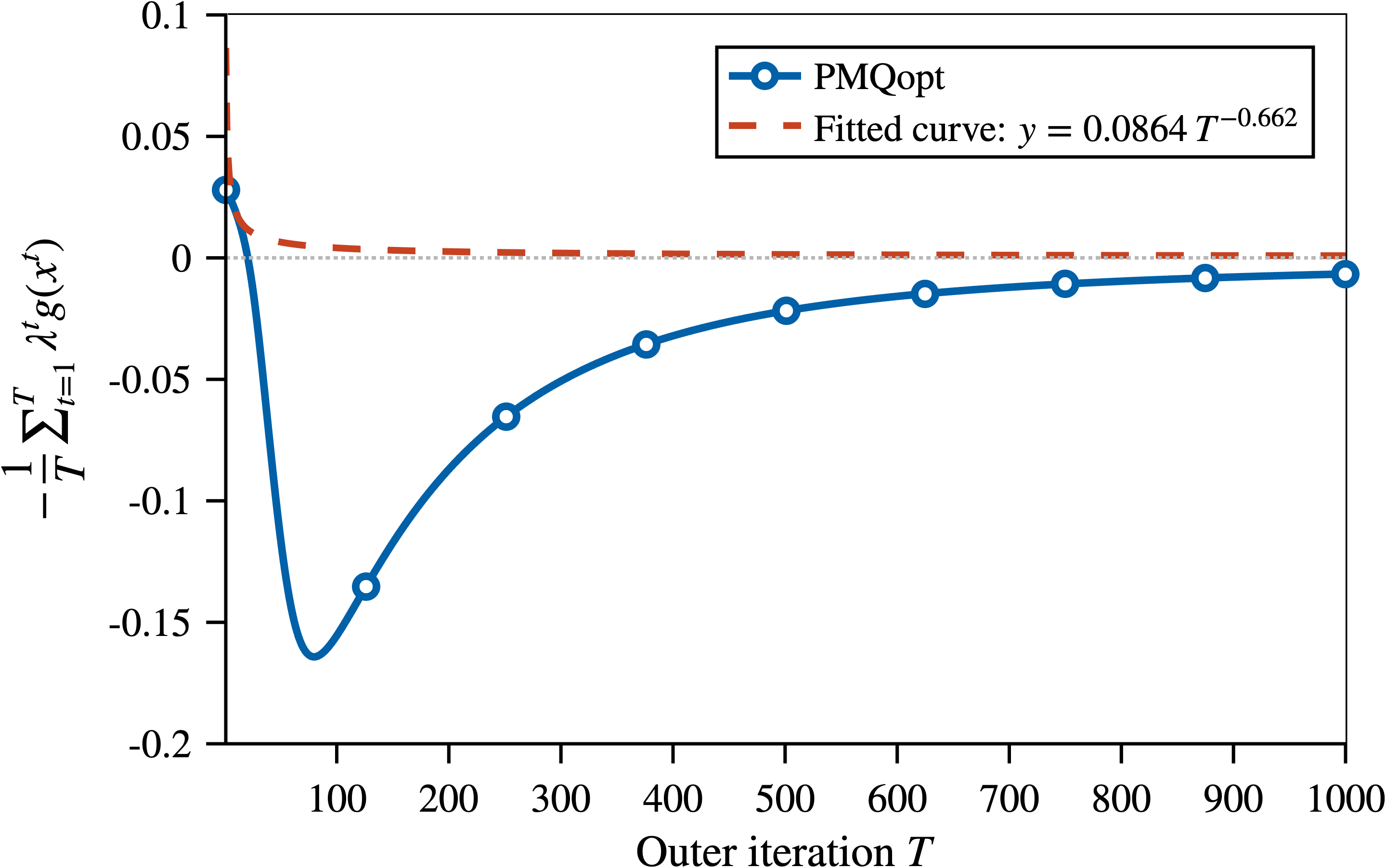}
        \label{fig:minist3}

    }
    \hfill
    \subfloat[Time-to-target comparison]{
        \includegraphics[width=0.48\textwidth]{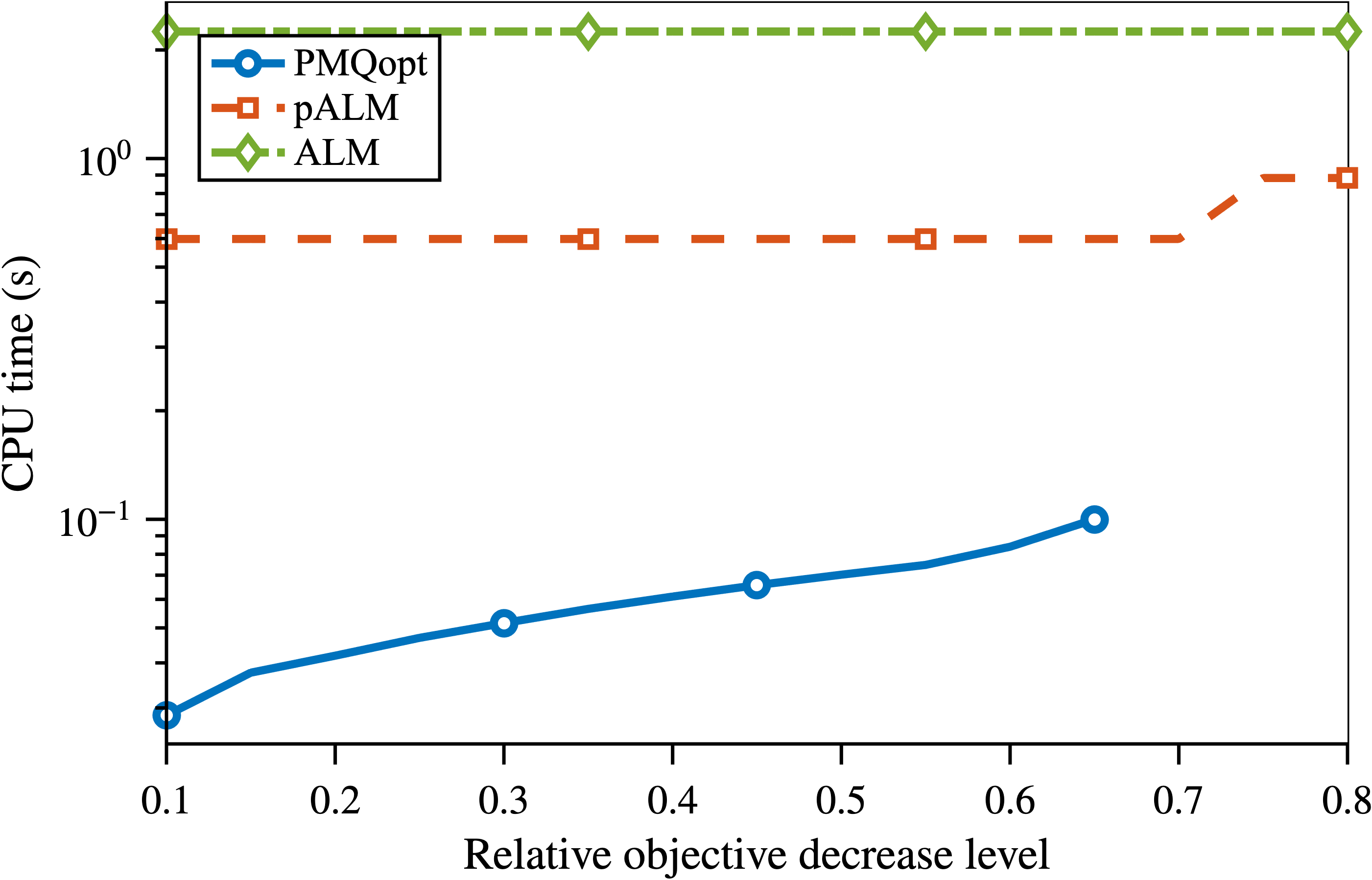}
        \label{fig:minist}

    }
    \caption{Theory validation and algorithmic
comparison on the MNIST dataset.}
    \label{minist}
\end{figure}

\section{Conclusion}\label{Sec5}
\setcounter{equation}{0}
In this paper, we present a proximal augmented Lagrangian  method for solving
 inequality constrained nonconvex  optimization problems whose functions are smooth weakly convex. The numerical method QPALM is designed based on solving a QCQP approximation of the original problem at each iteration. We  demonstrated the result about the iteration complexities of QPALM in terms of the squared norm of the gradient of the Moreau envelope of the Lagrangian, the average constraint violation, and the average complementarity violation.  Numerical results demonstrate
the superiority in comparison with pALM and ALM in the literature. For general non-convex optimization problems, under weaker conditions than those in this paper, how to analyze the iteration complexities of QPALM is still quite challenging, however this is an interesting topic left. 

\begin{thebibliography}{99}


\bibitem{Adeoye2025}
A. D. Adeoye, P. Latafat, and A. Bemporad,{\sl
A proximal augmented Lagrangian method for nonconvex optimization with equality and inequality constraints},arXiv preprint,2025, arXiv:2509.02894.
\bibitem{Beck2017}
Amir Beck, {\sl First-Order Methods in Optimization}, SIAM, Philadelphia, 2017.

 \bibitem{Bertsekas1982}
D.P.
Bertsekas, {\sl Constrained Optimization and Lagrange
Multiplier Methods},  Academic Press, New York, 1982.
\bibitem{Curtis2024}
Frank E. Curtis, Michael J. O'Neill and Daniel P. Robinson, {\sl
Worst-case complexity of an SQP method for nonlinear
equality constrained stochastic optimization}, Mathematical Programming, 205(2024),431-483.
\bibitem{Xuyy2025}
Hari Dahal, Wei Liu and Yangyang Xu,{\sl Damped proximal augmented Lagrangian method for
weakly-Convex problems with convex constraints}, Mathematical Programming Computation
https://doi.org/10.1007/s12532-025-00302-1.
\bibitem{GrapigliaYuan2021}
G. N. Grapiglia and Y. Yuan,{\sl
On the complexity of an augmented Lagrangian method for nonconvex optimization},
IMA Journal of Numerical Analysis,
41:2(2021), 1546--1568.
\bibitem{Guyon2004}
Isabelle Guyon, Steve Gunn, Aviel Ben-Hur, Gideon Dror,{\sl
Result analysis of the NIPS 2003 feature selection challenge},
Advances in Neural Information Processing Systems, Vol.
17, 2004, pp. 545--552 (MIT Press).



\bibitem{HajinezhadHong2019}
D. Hajinezhad and M. Hong,{\sl
Perturbed proximal primal-dual algorithm for nonconvex nonsmooth optimization},
Mathematical Programming,
176:1-2(2019), 207-245.
\bibitem{Hestenes1969}
M. R. Hestenes,{\sl
Multiplier and gradient methods},
Journal of Optimization Theory and Applications,
 4:5(1969), pp. 303--320.


\bibitem{HongHajinezhadZhao2017}
M. Hong, D. Hajinezhad, and M.-M.
Zhao,{\sl
Prox-PDA: the proximal primal-dual algorithm for fast distributed nonconvex optimization and learning over networks},
Proceedings of the 34th International Conference on Machine Learning,
Proceedings of Machine Learning Research, Vol.
70, 2017, pp. 1529--1538.

\bibitem{LeCun2010}
Yann LeCun, Corinna Cortes, Christopher J.C. Burges,
The MNIST Database of Handwritten Digits, 2010.
\url{http://yann.lecun.com/exdb/mnist/}.


\bibitem{P69}
 M.J.D.
Powell, {\sl A method for nonlinear constraints in
minimization problems}, In  R. Fletcher, Editor, {\sl
Optimization}, Academic Press, New York, 1969, 283--298.
\bibitem{Rock70}
 R. Tyrrell Rockafellar, {\sl Convex Analysis}, Princeton University Press,1970.
\bibitem{RW98}
R. T. Rockafellar and R. J. -B.
Wets, {\sl Variational  Analysis}, Springer-Verlag, New York, 1998.
\bibitem{CINA2008}
Workbench Team C,
A Marketing Dataset, 2008.
\url{http://www.causality.inf.ethz.ch/data/CINA.html}.
\bibitem{XieWright2021}
Y. Xie and S. J. Wright,{\sl
Complexity of Proximal Augmented Lagrangian for Nonconvex Optimization with Nonlinear Equality Constraints},
Journal of Scientific Computing,
86:38(2021), https://doi.org/10.1007/s10915-021-01409-y.
%
130, 2021, pp. 2170--2178.
\bibitem{Wang2015a}
X. Wang and Y.X. Yuan, An augmented Lagrangian trust region method for equality constrained optimization. Optim.
Methods Software,
30:3(2015), 559--582.
\bibitem{Wang2015b}
X. Wang and H. C. Zhang, An augmented Lagrangian affine scaling method for nonlinear programming. Optim.
Methods Software, 30:5(2015),934-964.




%
\bibitem{Zhang2023}
L. W. Zhang, Y. L. Zhang, X. T. Xiao and J. Wu,
Stochastic Approximation Proximal Method of Multipliers
for Convex Stochastic Programming,
Mathematics of Operations Research, 48:1(2023),177-193.
%
%
%
%
\end{thebibliography}
\end{document}